\documentclass{article}

\usepackage{amsmath,amssymb,amsthm}
\usepackage{halloweenmath}
\usepackage[T1]{fontenc}
\usepackage{kpfonts, baskervald}
\usepackage{xcolor}
\definecolor{darkgreen}{HTML}{228B22}
\newtheorem{thm}{Theorem}[section]
\newtheorem{cor}[thm]{Corollary}
\newtheorem{prop}[thm]{Proposition}
\newtheorem{lem}[thm]{Lemma}

\theoremstyle{definition}
\newtheorem{defn}[thm]{Definition}
\newtheorem*{defn*}{Definition}
\theoremstyle{remark}
\newtheorem{rmk}[thm]{Remark}
\newtheorem{exam}[thm]{Example}

\usepackage{microtype}

\usepackage[all]{xy}
\usepackage{tikz-cd}

\usepackage{todonotes}

\usepackage{hyperref}
\hypersetup{
  colorlinks   = true, 
  urlcolor     = blue, 
  linkcolor    = blue, 
  citecolor   = red 
}

\newcommand{\co}{\colon\thinspace}
\newcommand{\mb}[1]{\mathbb{#1}}
\newcommand{\mf}[1]{\mathfrak{#1}}

\newcommand{\xx}{\overline{\xi}}

\newcommand{\op}{{op}}

\newcommand{\too}{\xrightarrow}
\newcommand{\from}{\leftarrow}

\newcommand{\Pres}{\mathcal{P}}

\newcommand{\cC}{\mathcal{C}}
\newcommand{\dD}{\mathcal{D}}
\newcommand{\eE}{\mathcal{E}}

\newcommand{\xX}{\mathcal{X}}

\newcommand{\wt}[1]{\widetilde{#1}}

\newcommand{\loday}[2]{\mathcal{L}^{#1}(#2)}

\DeclareMathOperator{\Hom}{Hom}
\DeclareMathOperator{\Map}{Map}
\DeclareMathOperator{\Tor}{Tor}

\DeclareMathOperator{\Fun}{Fun}

\DeclareMathOperator{\Sp}{Sp}
\DeclareMathOperator{\Spaces}{\mathcal{S}}
\DeclareMathOperator{\sfin}{\mathcal{S}_\ast^{\text{fin}}}
\DeclareMathOperator{\sfinunpointed}{\mathcal{S}^{\text{fin}}}

\DeclareMathOperator{\Alg}{Alg}

\DeclareMathOperator{\CAlg}{CAlg}
\DeclareMathOperator{\LMod}{LMod}

\DeclareMathOperator{\Mod}{Mod}

\DeclareMathOperator{\Lift}{Lift}
\DeclareMathOperator{\id}{id}
\DeclareMathOperator{\Exc}{Exc}

\DeclareMathOperator*{\colim}{colim}
\DeclareMathOperator*{\hocolim}{hocolim}
\DeclareMathOperator*{\holim}{holim}

\DeclareMathOperator*{\fib}{fib}
\DeclareMathOperator*{\cofib}{cofib}
\DeclareMathOperator*{\fibprod}{\times}

\DeclareMathOperator{\THH}{THH}

\DeclareMathOperator{\MU}{MU}
\DeclareMathOperator{\BP}{BP}

\DeclareMathOperator*{\sma}{\wedge}
\DeclareMathOperator*{\tens}{\otimes}

\title{Skeleta and categories of algebras}
\author{Jonathan Beardsley, Tyler Lawson}

\begin{document}
\maketitle

\begin{abstract}
	We define a notion of a connectivity structure on an $\infty$-category, analogous to a $t$-structure but applicable in unstable contexts---such as spaces, or algebras over an operad. This allows us to generalize notions of n-skeleta, minimal skeleta, and cellular approximation from the category of spaces. For modules over an Eilenberg--Mac Lane spectrum, these are closely related to the notion of projective amplitude.
	
	We apply these to ring spectra, where they can be detected via the cotangent complex and higher Hochschild homology with coefficients. We show that the spectra $Y(n)$ of chromatic homotopy theory are minimal skeleta for $H\mathbb{F}_2$ in the category of associative ring spectra. Similarly, Ravenel's spectra $T(n)$ are shown to be minimal skeleta for $BP$ in the same way, which proves that these admit canonical associative algebra structures.
\end{abstract}

\tableofcontents

\section{Introduction}

Within homotopy theory, it is commonly understood (if not always made explicit) that the homology groups of a space $X$ are closely coupled with how $X$ can be built as a CW-complex. If $X$ is equivalent to a CW-complex, then $C_*(X)$ is quasi-isomorphic to a chain complex with one free generator for each cell in the CW-structure. For 1-connected spaces $X$ we can do better, because a converse holds: if $\wt{C}_*(X)$ is quasi-isomorphic to a complex that is levelwise free, then there exists a CW-complex equivalent to $X$ with one cell for each generator.

The engine that makes this technique possible is that we can understand the relation between connectivity and cell attachment. For maps of simply-connected spaces, homology detects connectivity. Further, suppose we have built the $n$-skeleton for a CW-approximation of $X$: an $n$-connected map $X^{(n)} \to X$. The set of possible ways to ``extend this map to an $(n+1)$-cell'' is governed by the relative homotopy group $\pi_{n+1}(X,X^{(n)})$; similarly, given a map of chain complexes $C \to D$, the set of possible ways to ``extend this map to an $(n+1)$-cell'' is governed by the relative homology group $H_{n+1}(D,C)$. The map
\[
  \pi_{n+1}(X,X^{(n)}) \to H_{n+1}(C_*X,C_* X^{(n)})
\]
is an isomorphism by the relative Hurewicz theorem, making it possible to lift the attaching map for each $(n+1)$-dimensional generator uniquely from the chain level to the space level.

With $1$-connected rational or $p$-complete spaces, the same argument works but our job is easier: there exists a rational or $p$-complete CW-structure on $X$ with one cell for each element in a basis of $H_*(X;\mb Q)$ or $H_*(X;\mb F_p)$ respectively. This argument also works for connective spectra, and this is a core component of understanding objects in \emph{stable} homotopy theory via their homology. This perspective also makes it clear that, because the Blakers--Massey excision theorem can be used to give an isomorphism
\[
  \pi_n(X,X^{(n)}) \to \pi_{n+1}(\Sigma X, \Sigma X^{(n)}),
\]
it is possible to lift (up to homotopy equivalence) any CW-structure from $\Sigma X$ to $X$.

This paper studies these types of cell attachments in other categories. Our main application is to demonstrate that certain spectra in chromatic homotopy theory have natural multiplicative structures: specifically, that they arise as ``skeleta'' in cellular constructions of known ring spectra. The desire to having an analogous theory of homology for commutative ring spectra, detecting cell attachment, was one motivation for the development of topological Andr\'e--Quillen homology by Kriz, Basterra, and Mandell \cite{kriz-towers, basterra-andrequillen, basterra-mandell-taqcohomology}. However, our first goal will be to understand exactly what it means to say ``$K$ is an $n$-skeleton of $X$'' in a rather broad generality.

To say that $K$ is an $n$-skeleton of $X$ in the terms above relies on the existence of a particular type of \emph{construction} of $X$ that produces $K$ at some stage. We hope to demonstrate that many of the useful homotopy-theoretic properties of a skeleton $K$ can be derived from the definition of connectivity of a map.

\begin{defn*}
  An object $K$ is \emph{$k$-skeletal} if it has the left lifting property with respect to $k$-connected maps: for any $k$-connected map $X \to Y$, every map $K \to Y$ lifts to a map $K \to X$.

  A map $f\co K \to X$ makes $K$ into a \emph{$k$-skeleton} of $X$ if $K$ is $k$-skeletal and $f$ is $k$-connected.
\end{defn*}

For example, a space is $k$-skeletal if and only if (in the homotopy category) it is a retract of a $k$-dimensional CW-complex (cf.~Example \ref{example: retracts of kdim CW}). However, this definition is cell-free, interacts well with adjunctions, and has well-behaved generalizations even if we adopt a flexible notion of what ``connectivity'' means. When we \emph{can} construct objects inductively with cells, the stages will naturally be skeleta. More, because the definition of skeletal objects is in terms of lifting properties, we can often detect $k$-skeletal objects by testing analogues of cohomology.

This will allow us to give criteria under which the construction of a skeleton can be lifted from one category to another. Our engine for lifting CW-structures in terms of relative homotopy and homology is a special property of an \emph{adjunction} between two categories. The general notion that we introduce is that of a \emph{subductive adjunction}: one where a $k$-connected map $X \to Y$ induces a $k$-Cartesian square
\[
  \xymatrix{
    X \ar[r] \ar[d] & Y \ar[d] \\
    RL(X) \ar[r] & RL(Y).
  }
\]
Subductive adjunctions capture the situation where attaching maps for cells can be lifted along the functor $L$. If, instead, $k$-connected maps are taken to $(k+1)$-Cartesian squares, we have a \emph{strongly subductive adjunction}, where attaching maps lift uniquely.\footnote{The term ``subductive'' is loosely borrowed from geology; we think of the adjunction as an interaction between two tectonic plates, with the left adjoint $L\co \cC \to \dD$ gradually pushing objects from $\cC$ under the surface. Even though this process can (and typically does) lose information from the category $\cC$, it is still gradual enough that (with enough effort) we can gradually dig up information about objects and maps to reassemble them.}

\subsection{Main applications}

The proofs of the nilpotence and periodicity theorems by Devinatz--Hopkins--Smith make heavy use of ring spectra $X(n)$, built as Thom spectra on the spaces $\Omega SU(n)$. These assemble into a directed system
\[
  \mb S = X(1) \to X(2) \to X(3) \to \dots \to \MU
\]
of ring spectra (specifically, $\mb E_2$ ring spectra) and on integral homology this realizes a filtration of $H_* \MU$ by polynomial subalgebras,
\[
  \mb Z \to \mb Z[x_1] \to \mb Z[x_1,x_2] \to \dots \to H_* \MU
\]
where $|x_i| = 2i$. When working $p$-locally, however, $\MU_{(p)}$ has a split summand $\BP$ that is more tractable for computational applications, and the filtration of $\MU$ by $X(n)$ can be replaced by a filtration
\[
  \mb S_{(p)} = T(0) \to T(1) \to T(2) \to \dots \to \BP.
\]
On homology this realizes a filtration of $H_* \BP$,
\[
  \mb Z_{(p)} \to \mb Z_{(p)}[t_1] \to \mb Z_{(p)}[t_1,t_2] \to \dots \to H_* \BP,
\]
where $|t_i| = 2p^i-2$. However, the spectrum $T(n)$ is usually constructed as a split summand of $X(p^n)_{(p)}$ and this splitting does not grant it structured multiplication.

We will show that this sequence can be realized in a \emph{canonical} way as a sequence of associative ring spectra. We know that $\BP$ exists and has an associative structure, and we can construct a cellular approximation by attaching cells inductively---but within the category of \emph{associative ring spectra}, rather than just within the category of spectra. We will show that the $T(n)$ arise quite naturally as skeleta of $\BP$ as an associative algebra, and that the necessary cells are detected by topological Hochschild homology with coefficients. (Cellular constructions, in various categories, of objects related to $\BP$ are by no means new: see, for example, work of Priddy \cite{priddy-cellularBP}, Hu--Kriz--May \cite{hu-kriz-may-cores}, or Baker \cite{baker-closeencounters}.) By similar methods, it is also possible to show that the $2$-primary spectra $Y(n)$ of Mahowald--Ravenel--Shick \cite{mahowald-ravenel-shick-telescope} arise as skeleta of the Eilenberg--Mac Lane spectrum for $\mb F_2$, although this can be shown more directly using their construction as Thom spectra.

\subsection{Outline}

In Section~\ref{sec:connectivity} we will introduce the notion of a \emph{connectivity structure} on $\cC$, generalizing the standard notion of ``$k$-connectedness'' for spaces and spectra, and prove several properties. In the case where $\cC$ is stable, the most common way to get a connectivity structure is to have a $t$-structure on $\cC$. Several of our categories of interest, such as categories of algebras, are not stable, but they still inherit connectivity structures from a forgetful functor to an underlying stable category.

Section~\ref{sec:skeleta} introduces \emph{skeleta} and \emph{skeletal objects} in $\cC$, defined in terms of lifting against sufficiently connected maps, and proves basic properties. We will also discuss the notion of \emph{cells} in $\cC$ and give a version of CW-approximation when there is a sufficient supply of cells. We will introduce \emph{minimal skeleta}, which (when they exist) are unique up to homotopy equivalence.

Section~\ref{sec:detectingskeleta} discusses the \emph{detection} of skeletal objects. The chief application is to show that tools like stabilization can be used to detect skeletality so long as $k$-connected maps admit a Postnikov-like decomposition.

Section~\ref{sec:lifting}, by contrast, discusses \emph{lifting} skeleta across an adjunction with left adjoint $L$. We introduce subductive adjunctions and how they allow analogues of lifting CW-structures: inductively, we can lift attaching maps of cells along $L$ from the target category to the source category. This allows us to obtain results lifting the construction of skeleta, and minimal skeleta, along $L$.

Section~\ref{sec:excision} discusses \emph{excision} theorems, and how their presence often makes stabilization into a subductive adjunction. This makes it possible to employ homology theories for ring spectra, such as those introduced by Basterra and Mandell \cite{basterra-mandell-iteratedthh}.

Finally, Section~\ref{sec:applications} discusses our applications. For the derived category of a ring, we find that skeletality is closely related to the notion of \emph{projective amplitude}, and for spaces we find that $k$-skeletality is equivalent to a constraint on the projective amplitude for chains on the universal cover.

We then work with algebra spectra, showing that cellular constructions of associative algebras are governed by topological Hochschild homology. For these, we need to make heavy use Ching--Harper's results on excision for categories of algebras over an operad.

\subsection{Notation and conventions}

Throughout this paper, we make the assumption that $\cC$ is an $\infty$-category in the sense of \cite{lurie-htt} that admits pullbacks. We write $h\cC$ for the associated homotopy category.

\subsection{Acknowledgements}

The authors would like to thank Jeremy Hahn for being a close part of questions and discussions that instigated this project, and 
Clark Barwick,
Tim Campion,
Ian Coley,
and
Denis Nardin
for help with material related to this paper. We also thank an anonymous referee for a number of helpful clarifications. 

This work is an outgrowth of conversations that took place while the authors were attending the workshop ``Derived algebraic geometry and chromatic homotopy theory'' during the \textit{Homotopy Harnessing Higher Structures} program at the Isaac Newton Institute (EPSRC grant numbers EP/K032208/1 and EP/R014604/1). The first author's travel to the workshop was funded by NSF grant DMS-1833295. The first author was also partially supported by NSF grant DMS-1745583 and Simons Foundation Collaboration Grant \#853272 while working on this paper. The second author was partially supported by NSF grant DMS-1610408.

\section{Connectivity structures}
\label{sec:connectivity}

\subsection{Definitions}

\begin{defn}
  \label{def:connectivity}
  Suppose that $\cC$ is an $\infty$-category with homotopy pullbacks. A \emph{connectivity structure} $T$ on $\cC$ consists of, for each $k \in \mb Z$, a collection of \emph{$k$-connected maps} in the homotopy category $h\cC$, subject to the following axioms.
  \begin{enumerate}
  \item A map that is $k$-connected is also $(k-1)$-connected.
  \item Equivalences are $k$-connected for all $k$.
  \item The homotopy pullback of a $k$-connected map is $k$-connected.
  \item Suppose we have composable maps $f\co X \to Y$ and $g\co Y \to Z$.
    \begin{enumerate}
    \item If $f$ and $g$ are $k$-connected, so is $gf$.
    \item If $f$ is $(k-1)$-connected and $gf$ is $k$-connected, then $g$ is $k$-connected.
    \item If $gf$ is $k$-connected and $g$ is $(k+1)$-connected, then $f$ is $k$-connected.
    \end{enumerate}
  \end{enumerate}
  When specifying the connectivity structure is necessary, we refer to a map as $k$-connected in $T$.

  A functor $F\co \cC \to \dD$ \emph{preserves connectivity} if, whenever $f$ is $k$-connected, the image $F(f)$ is $k$-connected.
\end{defn}

\begin{rmk}
  For convenience, we will refer to $f$ as $\infty$-connected if it is $k$-connected for all $k$. Similarly, we will refer to all maps as $(-\infty)$-connected. We define a connectivity structure to be \emph{left complete} if an $\infty$-connected map is an equivalence.\footnote{This is by analogy with $t$-structures: see \S\ref{sec:stableconnectivity}.}
\end{rmk}

The definition of a connectivity structure is invariant under translation.
\begin{defn}
  For $d \in \mb Z$, the connectivity structure $(T+d)$ obtained by \emph{shifting $T$ by $d$} is defined as follows. A map $f$ is $k$-connected in $T$ if and only if $f$ is $(k+d)$-connected in $(T+d)$.

  A functor $F\co \cC \to \dD$ \emph{adds (at least) $d$ to connectivity} if, whenever $f$ is $k$-connected, the image $F(f)$ is $(k+d)$-connected.\footnote{In other words, it is a connectivity-preserving functor $(\cC,T_\cC+d) \to (\dD, T_\dD)$.}
\end{defn}

\begin{defn}
  \label{def:intersectionstructure}
  If $\{T_i\}$ is a collection of connectivity structures, the \emph{intersection connectivity structure} $\cap T_i$ is defined as follows. A map $f$ is $k$-connected in $\cap T_i$ if and only if, for all $i$, $f$ is $k$-connected in $T_i$.
\end{defn}

\begin{defn}
  \label{def:liftedstructure}
  Suppose that $R\co \dD \to \cC$ is a functor that preserves homotopy pullbacks and that $\cC$ has a connectivity structure $T$. Then there is a connectivity structure $R^{-1}(T)$ on $\dD$, defined as follows. A map $f\co X \to Y$ in $\dD$ is $k$-connected in $R^{-1}(T)$ if and only if the image $Rf$ is $k$-connected in $T$. We refer to this as the connectivity structure \emph{lifted} from $\cC$, and say that the functor $R$ \emph{reflects connectivity}.
\end{defn}

This is possible whenever $R$ is a right adjoint, such as a forgetful functor. The lifted structure is always the maximal structure such that $R$ preserves connectivity.

\begin{defn}
  If $\cC$ has homotopy pushouts, a connectivity structure on $\cC$ is \emph{compatible with cobase change} if the homotopy pushout of any $k$-connected map is $k$-connected.

  Suppose that $\cC$ has $K$-indexed homotopy colimits. A connectivity structure on $\cC$ is \emph{compatible with $K$-indexed (homotopy) colimits} if, whenever a natural transformation $F \to G$ of $K$-indexed diagrams is $k$-connected objectwise, the map $\hocolim_K F \to \hocolim_K G$ is $k$-connected.
\end{defn}

\subsection{Examples}

\begin{exam}
  Every $\cC$ has available the \emph{minimal} connectivity structure $T_{min}$, where only equivalences are $k$-connected for $k \neq -\infty$. It also has the \emph{maximal} connectivity structure $T_{max}$, where all maps are $\infty$-connected.

  Every functor out of the minimal connectivity structure preserves connectivity; every functor into the maximal connectivity structure preserves connectivity.
\end{exam}

\begin{exam}
  \label{exam:spaces}
  A map $f\co X \to Y$ in the category $\Spaces$ of spaces is $k$-connected if (for all choices of basepoint) $\pi_j(f)$ is an isomorphism for all $j<k$, and $\pi_k(f)$ is a surjection.
\end{exam}

\begin{exam}\label{exam:spectratstructure}
  A morphism $f\co X\to Y$ of objects in the category $\Sp$ of spectra is called $k$-connected if $\pi_j(f)$ is an isomorphism for all $j<k$, and $\pi_j(f)$ is a surjection.

  More generally, if $\cC$ is stable and has a $t$-structure, we can give $\cC$ a connectivity structure by declaring that a map $f\co X \to Y$ is $k$-connected if its cofiber $\cofib(f)$ is $k$-connected: $\tau_{\leq k} \cofib(f)$ is trivial (cf.~Sections \ref{sec:stableconnectivity} and \ref{sec:stabilization}).
\end{exam}

\begin{rmk}
By \cite[1.2.1.6]{lurie-higheralgebra}, if $\cC$ is stable and the connectivity structure on $\cC$ is determined by a $t$-structure, then the connectivity structure is compatible with cobase change and all homotopy colimits that exist in $\cC$.
\end{rmk}

\begin{exam}
  If $\cC_i$ have connectivity structures, then the product $\prod_i \cC_i$ has a \emph{product connectivity structure}: a map $(X_i) \to (Y_i)$ is $k$-connected when each map $X_i \to Y_i$ is. This is the intersection of the connectivity structures lifted from each $\cC_i$.
\end{exam}

\begin{exam}
  If $\mf O$ is an (ordinary) operad, the category of $\mf O$-algebras in $\cC$ (if it has pullbacks) has a connectivity structure lifted from $\cC$: the forgetful functor preserves all homotopy limits that exist in $\cC$.
\end{exam}

\begin{exam}\label{exam:slicecoslice}
  If $\cC$ has a connectivity structure, then the slice and coslice categories $\cC_{/Y}$ and $\cC_{X/}$ have a connectivity structure lifted from $\cC$.
\end{exam}

\begin{exam}\label{exam:topoi}
  Suppose $\xX$ is an $\infty$-topos in the sense of \cite{lurie-htt}. Then $\xX$ admits a connectivity structure where a morphism is $k$-connected precisely when it is $(k+1)$-connective in the sense of \cite[6.5.1.10]{lurie-htt}. Specifically, a morphism $f\colon X\to Y$ in $\cC$ is $k$-connected if it is an effective epimorphism as in \cite[6.2.3.5]{lurie-htt} and if $\pi_i(f)\simeq\ast$ (for the toposic homotopy groups of \cite[6.5.1]{lurie-htt}) for all $0\leq i\leq k$. The first three axioms of Definition \ref{def:connectivity}, as well as part (a) of the fourth axiom, are satisfied as a result of \cite[Proposition 6.5.1.16]{lurie-htt}.
  
   For parts (b) and (c) of the fourth axiom, recall that by \cite[6.1.0.6]{lurie-htt} any $\infty$-topos $\xX$ is a left exact localization of a presheaf $\infty$-topos $\Pres(\cC)=Fun(\cC,\Spaces)$. If we say that $\xX\simeq L\Pres(\cC)$ for a localization functor $L$ then a morphism $f\in\xX$ is $k$-connected if and only if it is $k$-connected in $\Pres(\cC)$. One direction of this assertion follows from \cite[6.5.1.15]{lurie-htt}, i.e.~if $f\simeq L(f_0)$ is $k$-connected then so is $f_0$. For the other direction, notice that the left exact left adjoint $L$ preserves colimits and truncations (cf.~\cite[5.5.6.28]{lurie-htt}) so preserves $k$-connectedness (since $k$-connectedness in an $\infty$-topos can be determined by truncations \cite[6.5.1.12]{lurie-htt}). In the presheaf topos $\Pres(\cC)$, $k$-truncation, and therefore $k$-connectedness, is determined ``pointwise'' in $\Spaces$, the statements follow from their usual proofs in $\Spaces$. 
   
   Note that by \cite[3.6.6]{abfj--blakersmassey} the definition of a $k$-connected morphism in an $\infty$-topos $\xX$ described above agrees with the definition of a $k$-connected morphism given in Example \ref{exam:spaces} when $\xX\simeq\Spaces$.
\end{exam}  

\begin{rmk}
  There is a discrepancy between the definition of $k$-connected in an $\infty$-topos given in \cite{lurie-htt} and in our notion of a connectivity structure in that, in an $\infty$-topos, the concept of being $k$-connected only makes sense for $k\geq -2$. However we can extend the notion of $k$-connectedness in an $\infty$-topos described in Example \ref{exam:topoi} by saying that all morphisms are $k$-connected for any $k\leq -2$.
\end{rmk}

\subsection{Basic properties of connectivity}

\begin{prop}\label{prop:retract}
  Suppose that we have a retract diagram $A \to X \to A$. Then the map $A \to X$ is $k$-connected if and only if the map $X \to A$ is $(k+1)$-connected.
\end{prop}

\begin{proof}
  This follows immediately from the second axiom and items (b) and (c) of the fourth axiom of Definition \ref{def:connectivity}.
\end{proof}

\begin{prop}
  \label{prop:pullbackconnectivity}
  Suppose that we have a commutative diagram
  \[
    \xymatrix{
      X \ar[d] \ar[r] & Y \ar[d] & Z\ar[l] \ar[d] \\
      X'\ar[r] & Y' & Z' \ar[l]
    }
  \]
  where the middle vertical map is $(k+1)$-connected and the outer vertical maps are $k$-connected. Then the induced map of homotopy pullbacks $X \times_Y Z \to X' \times_{Y'} Z'$ is $k$-connected.
\end{prop}

\begin{proof}
  Consider the composite map
  \[
    X \times_Y Z \to X' \times_{Y'} Z \to X' \times_{Y'} Z'.
  \]
  We will show that both maps are $k$-connected, which implies the desired result.
  
  The map $X' \times_{Y'} Z \to X' \times_{Y'} Z'$ is the pullback of the $k$-connected map $Z \to Z'$ along the map $X' \to Y'$, and so it is $k$-connected.

  The map $X' \times_{Y'} Y \to X'$ is the pullback of a $(k+1)$-connected map $Y \to Y'$ along $X' \to Y'$, and so it is $(k+1)$-connected. The composite $X \to X' \times_{Y'} Y \to X'$ is $k$-connected. Therefore, the map $X \to X' \times_{Y'} Y$ is $k$-connected. Taking pullback along the map $Z \to Y$, we find that the map $X \times_Y Z \to X' \times_{Y'} Z$ is $k$-connected.
\end{proof}

By considering the special case where $X=Y=Z=X'=Z'$, we arrive at the following.
\begin{cor}
  \label{cor:diagonalconnectivity}
  If $f\co X \to Y$ is $(k+1)$-connected, then the induced map $\Delta\co X \to X \times_Y X$ is $k$-connected.
\end{cor}

\begin{rmk}
If $\cC$ is an $\infty$-topos then \cite[6.5.1.18]{lurie-htt} gives a partial converse to the above corollary. Specifically, if the map $\Delta$ is $k$-connected and also an effective epimorphism then $f$ is $(k+1)$-connected.
\end{rmk}

As another special case, we can take $Y=Y'$ and $Z=Z'$.

\begin{cor}\label{cor:pullbackpreservessliceconnectivity}
Let $f\co Z\to Y$ be a morphism in a category $\cC$ with a connectivity structure. Then the pullback functor $f^\ast\co \cC_{/Y}\to \cC_{/Z}$ preserves connectivity (for the connectivity structures lifted from $\cC$).
\end{cor}

The following is a dual argument to Proposition~\ref{prop:pullbackconnectivity}.
\begin{prop}
  \label{prop:pushoutconnectivity}
  Suppose that the connectivity structure on $\cC$ is compatible with cobase change, and that we have a commutative diagram
  \[
    \xymatrix{
      A \ar[d] & B \ar[d] \ar[l] \ar[r] & C \ar[d] \\
      A' & B' \ar[l] \ar[r] & C'
    }
  \]
  where the middle vertical map is $(k-1)$-connected and the outer vertical maps are $k$-connected. Then the induced map of homotopy pushouts $A \amalg_B C \to A' \amalg_{B'} C'$ is $k$-connected.
\end{prop}

\begin{proof}
  Consider the composite map
  \[
    A \amalg_B C \to A \amalg_B C' \to A' \amalg_{B'} C'.
  \]
  We will show that both maps are $k$-connected, which implies the desired result.
  
  The map $A \amalg_{B} C \to A \amalg_{B} C'$ is the homotopy pushout of the $k$-connected map $C \to C'$ along the map $B \to A$, and so it is $k$-connected.

  The map $A \to A \amalg_{B} B'$ is the homotopy pushout of a $(k-1)$-connected map $B \to B'$ along $B \to A$, and so it is $(k-1)$-connected. The composite $A \to A \amalg_{B} B' \to A'$ is $k$-connected. Therefore, the map $A \amalg_{B} B' \to A'$ is $k$-connected. Taking homotopy pushouts along the map $B' \to C'$, we find that the map $A \amalg_{B} C' \to A' \amalg_{B'} C'$ is $k$-connected.
\end{proof}

\begin{cor}
  \label{cor:amalgconnectivity}
  Suppose that $\cC$ has homotopy pushouts and that the connectivity structure on $\cC$ is compatible with cobase change. Given any map $A \to B$, the functor $B \coprod_A(-)\co \cC_{A/} \to \cC_{B/}$ preserves connectivity.
\end{cor}

\begin{cor}
  \label{cor:freudenthal}
  Suppose that $\cC$ has homotopy pushouts and that the connectivity structure on $\cC$ is compatible with cobase change. If $X \to Y$ is a $(k-1)$-connected map of objects over $Z$, then the induced map $Z \amalg_X Z \to Z \amalg_Y Z$ is $k$-connected.
\end{cor}

\subsection{Stable connectivity structures}\label{sec:stableconnectivity}

Suppose $\cC$ is pointed. Proposition~\ref{prop:pullbackconnectivity} shows that whenever $f\co X \to Y$ is $k$-connected, $\Omega f\co \Omega X \to \Omega Y$ is $(k-1)$-connected: the loop operator $\Omega$ adds at least $(-1)$ to connnectivity. In the case of a stable category, $\Omega$ is an autoequivalence, and it is common for connectivity structures to satisfy a converse.

\begin{defn}
  \label{def:stablestructure}
  Suppose that $\cC$ is stable. A \emph{stable} connectivity structure on $\cC$ is a connectivity structure such that $f\co X \to Y$ is $k$-connected if and only if $\Omega f\co \Omega X \to \Omega Y$ is $(k-1)$-connected.
\end{defn}

In the category of spectra, we can detect connectivity by examining connectivity of the cofiber. This property is always possible with stable connectivity structures.

\begin{defn}\label{def:connectedobjects}
Let $\cC$ be a pointed category with a connectivity structure. Then we say that $X\in\cC$ is $k$-connected if the map from the zero object $\ast \to X$ is $k$-connected.
\end{defn}

\begin{rmk}\label{rmk:zeroiskconnected}
Note that the zero object of a pointed category will be $\infty$-connected.
\end{rmk}

\begin{lem}\label{lem:connectedfiberinpointed}
If $\cC$ is pointed, then the fiber of any $k$-connected morphism is $(k-1)$-connected.
\end{lem}

\begin{proof}
Let $f\co C\to D$ be a $k$-connected morphism in $\cC$ and consider the pullback square:
\[
\xymatrix{
\fib(f)\ar[d]_{p}\ar[r] & C\ar[d]^f\\
\ast\ar[r] & D
}
\]
By pullback stability the morphism $p$ is $k$-connected, so the zero morphism $\ast\to\fib(f)$ is $(k-1)$-connected by Proposition~\ref{prop:retract}.
\end{proof}

The above definition allows us to give another characterization of the $k$-connected morphisms in a stable connectivity structure, since stable categories always have zero objects.
\begin{prop}
  \label{prop:stablecofiber}
  If $\cC$ has a stable connectivity structure, then a map $f\co X \to Y$ is $k$-connected if and only if the cofiber $\cofib(f)$ is $k$-connected.
\end{prop}

\begin{proof}
  Because $\cC$ is stable, we have the following two homotopy pullback squares:
  \[
    \xymatrix{
      X \ar[d]_f \ar[r] & \ast \ar[d] &
      \Omega \cofib(f) \ar[r] \ar[d] & X \ar[d]^f \\
      Y \ar[r] & \cofib(f) &
      \ast \ar[r] & Y
    }
  \]
  The left-hand pullback, along with pullback stability of connectivity, shows that if $\cofib(f)$ is $k$-connected, then the map $f$ is $k$-connected.

  For the converse, consider the right-hand pullback. It shows that if $f$ is $k$-connected, then $\Omega \cofib(f)$ is $(k-1)$-connected by Lemma~\ref{lem:connectedfiberinpointed}. By stability, the map $\ast \to \cofib(f)$ is $k$-connected.
\end{proof}

\begin{rmk}
Stability of $\cC$ and Proposition \ref{prop:stablecofiber} of course immediately also imply that a map $f\co X\to Y$ is $k$-connected if and only if $\fib(f)$ is $(k-1)$-connected, giving a converse to Lemma \ref{lem:connectedfiberinpointed}. 
\end{rmk}

\begin{rmk}\label{rmk:stableconnectivitystructuregeneratedbyconnected}
Given a stable connectivity structure on $\cC$ we can define the subcategory of connective objects, $\cC_{\geq 0}$, as the full subcategory of all objects which are $(-1)$-connected. It is not hard to check that $\cC_{\geq 0}$ satisfies the following conditions:
\begin{enumerate}
  \item $\cC_{\geq 0}$ contains a zero object;
  \item $\cC_{\geq 0}$ is closed under extensions; and
  \item $\cC_{\geq 0}$ is closed under cofibers.
  \end{enumerate}
 Conversely, given a stable $\infty$-category $\cC$, a full subcategory $\dD\subseteq\cC$ determines a stable connectivity structure on $\cC$ so long as it satisfies these three conditions. The third condition can be replaced by closure under the suspension operator $\Sigma$, because a cofiber sequence $X \to Y \to \cofib(f)$ is equivalent data to an extension $Y \to \cofib(f) \to \Sigma X$.
 
More specifically, we can say that a morphism $f\co X\to Y$ is $k$-connected, for each $k\in\mb Z$, exactly when $\cofib(f)\in \Sigma^{k+1}\dD$. Checking that this determines a connectivity structure on $\cC$ is straightforward for the first three axioms. All three parts of the fourth axiom can be checked by recalling that if $\cC$ is stable then $h\cC$ is triangulated and applying the octahedral axiom.

To see that this connectivity structure is stable, notice that the suspension functor $\Sigma$ is an equivalence, so preserves cofibers. Thus  $f\co X\to Y$ is $k$-connected, i.e.~$\cofib(f)\in\Sigma^{k+1}\dD$, if and only if $\cofib(\Sigma^{-1}f)\simeq \Sigma^{-1}\cofib(f)\in\Sigma^k\dD$, i.e.~the desuspended morphism $\Sigma^{-1}f\co \Sigma^{-1}X\to \Sigma^{-1}Y$ is $(k-1)$-connected.  It also follows immediately that an object $X\in \cC$ is $k$-connected in the sense of Definition \ref{def:connectedobjects} if and only if $X\in\Sigma^{k+1}\dD$.
\end{rmk}

\begin{exam}
Let $E\in \Sp$ be a spectrum. Then there is a connectivity structure on $\Sp$ determined by declaring the subcategory of connective objects to be precisely those spectra $X$ such that the $E$-homology groups $E_i(X) = \pi_i(E\otimes X)$ are trivial for all $i<0$. 
\end{exam}

\subsection{Stabilization}\label{sec:stabilization}

Recall that a functor is reduced if it preserves terminal objects and excisive if takes pushout squares to pullback squares. If $\cC$ is an $\infty$-category with finite limits then there is a category $\Sp(\cC)$ of \emph{spectrum objects of $\cC$} \cite[1.4.2.8]{lurie-higheralgebra}. Explicitly, $\Sp(\cC)$ is the full subcategory $\Exc_\ast(\sfin,\cC)\subseteq \Fun(\sfin,\cC)$ of \textit{reduced, excisive} functors from finite pointed spaces to $\cC$. 

\begin{rmk}
For any $\infty$-category $\cC$ with finite limits there are functors $ev_d\co\Sp(\cC)\to \cC$ for all $d\geq 0$ given by evaluation on the $d$-spheres $S^d\in\sfin$ as in \cite[1.4.2.20]{lurie-higheralgebra}. For $d>0$ functors can be equivalently defined by composing $ev_0$ with $\Sigma^d$.  We will often write $\Omega^\infty$ instead of $ev_0$.
\end{rmk} 

\begin{defn}\label{def:connectedobjectsinstableconnectivitystructure}
Let $\cC$ be an $\infty$-category with finite limits and a connectivity structure, and let $\Sp(\cC)$ be its associated category of spectra. Say that an object $X$ of $\Sp(\cC)$ is \textit{connective} if for each $d\geq 0$, $ev_d(X)$ is $(d-1)$-connected in $\cC_\ast$. Let $\Sp(\cC)_{\geq 0}$ denote the full subcategory of $\Sp(\cC)$ spanned by the connective objects. 
\end{defn}

\begin{prop}
If $\cC$ has finite limits and a connectivity structure then $\Sp(\cC)_{\geq 0}\subseteq \Sp(\cC)$ determines a stable connectivity structure on $\Sp(\cC)$ as in Remark \ref{rmk:stableconnectivitystructuregeneratedbyconnected}. 
\end{prop}

\begin{proof}
  We only need to check the three conditions of Remark \ref{rmk:stableconnectivitystructuregeneratedbyconnected}.

  First, the zero object of $\Sp(\cC)$ is the functor given by $K\mapsto\ast$ for every $K\in\sfin$. It follows from Remark \ref{rmk:zeroiskconnected}
that its image is contained in $\Sp(\cC)_{\geq 0}$. 

Now let $X\to Y\to Z$ be a fiber sequence in $\Sp(\cC)$ with $X,Z\in\Sp(\cC)_{\geq 0}$, and fix $d\geq 0$. Then we have a pullback square
\[
\xymatrix{
X(S^d)\ar[r]\ar[d] & Y(S^d)\ar[d]\\
\ast\ar[r] & Z(S^d),
}
\]
in which the bottom horizontal morphism is $(d-1)$-connected. By pullback stability we have the top horizontal morphism is also $(d-1)$-connected. Because $X(S^d)$ is $(d-1)$-connected, the zero map $\ast\to X(S^d)$ is $(d-1)$-connected, so the composite $\ast\to X(S^d)\to Y(S^d)$ is $(d-1)$-connected, and therefore $Y(S^d)$ is $(d-1)$-connected.

Finally, we show that $\Sp(\cC)_{\geq 0}$ is closed under the shift operator $\Sigma$. Recall that for a reduced excisive functor $X\in\Sp(\cC)$ we have $\Sigma X(K)\simeq X(K\wedge S^1)$ (this follows for instance from \cite[1.4.2.13 (2)]{lurie-higheralgebra}). If $X(S^d)$ is $(d-1)$-connected for all $d$, then $(\Sigma X)(S^d)\simeq X(S^{d+1})$ is $d$-connected (and therefore $(d-1)$-connected) for all $d$.
\end{proof}

Recall the following definition from \cite[6.1.1.6]{lurie-higheralgebra}:

\begin{defn}\label{def:differentiable}
An $\infty$-category $\cC$ is \textit{differentiable} if it satisfies the following conditions:
\begin{enumerate}
\item $\cC$ admits finite limits;
\item $\cC$ admits sequential colimits; and
\item sequential colimits commute with finite limits.
\end{enumerate}
\end{defn}

\begin{lem}\label{lem:sigmainfinity}
  Suppose that $\cC$ is differentiable and has finite colimits. Then the functor $\Omega^\infty\co \Sp(\cC)\to \cC$ has a left adjoint $\Sigma^\infty$.
\end{lem}

\begin{proof}
  This is \cite[6.2.3.16]{lurie-higheralgebra}, but we give an explicit description of the functor $\Sigma^\infty$ for future use. The functor described in \cite{lurie-higheralgebra} is a composite. An object $C\in\cC_\ast$ determines a functor $f_C\in \Fun(\sfinunpointed,\cC_\ast)$ by declaring $f_C(\ast)\simeq C$ and extending by homotopy colimits. This may be extended to a functor $f_C^+\in \Fun(\sfin,\cC)$ by setting
  \[
    f_C^+(K)\simeq f_C(K)\amalg_C \ast
  \]
  where the morphism $C \simeq f_C(\ast) \to f_C(K)$ is determined by the base point of $K$. This gives a functor $\cC_\ast\to \Fun_\ast(\sfin,\cC)$, left adjoint to evaluation on $S^0$. There is an inclusion $\Exc_\ast(\sfin,\cC)\hookrightarrow \Fun_\ast(\sfin,\cC)$ whose left adjoint is given by the excisive approximation $F\mapsto P_1F$ (cf. the proof of \cite[6.2.1.12]{lurie-higheralgebra}). Thus the functor $\Sigma^\infty$ takes $C$ to the functor $P_1f_C^+$.
\end{proof}

\begin{rmk}\label{rmk:pointedstabilization}
When $\cC$ is also pointed, \cite[6.1.1.28]{lurie-higheralgebra} gives an equivalence $P_1F \simeq  \colim_{n}\Omega^n\circ F\circ \Sigma^n$, where $\Omega^n$ is formed in $\cC$ and $\Sigma^n$ is formed in $\sfin$. Thus $\Sigma^\infty$ is the composite functor \[ C\mapsto \colim_n\left(\Omega^n \circ f_C^+\circ \Sigma^n\right).\] The fact that $\cC$ is pointed also implies an equivalence $f_C^+(S^d)\simeq \Sigma^d C$ giving the familiar formula \[(\Sigma^\infty C)(S^d) \simeq \colim_n \Omega^n \Sigma^{n+d} C.\] In particular, $\Omega^\infty\Sigma^\infty C\simeq \colim_n\Omega^n\Sigma^n C$. 
\end{rmk}

\begin{prop}\label{prop:stabilizationpreservesconnective}
  Suppose that $\cC$ is pointed, differentiable, and has finite colimits, and that the connectivity structure on $\cC$ is compatible with cobase change and sequential colimits. Then the stabilization functor $\Sigma^\infty\co \cC \to \Sp(\cC)$ preserves connectivity.
\end{prop}

\begin{proof}
  By Proposition \ref{prop:stablecofiber} it suffices to show that if $g\co C\to D$ is $k$-connected in $\cC$ then $\cofib(\Sigma^\infty g)$ is $k$-connected in $\Sp(\cC)$. By stability, we may equivalently show that $\fib(\Sigma^\infty g)\simeq \Omega\cofib(\Sigma^\infty g)$ is $(k-1)$-connected. To say that $\fib(\Sigma^\infty g)$ is $(k-1)$-connected is equivalent to saying that $\Omega^{k}\fib(\Sigma^\infty g)\in\Sp(\cC)_{\geq 0}$, i.e.~$ev_d\Omega^k\fib(\Sigma^\infty g)$ is $(d-1)$-connected in $\cC$ for all $d\geq 0$. But from \cite[1.4.2.20]{lurie-higheralgebra} we have that
  \[
    ev_d\Omega^k\fib(\Sigma^\infty g)\simeq \Omega^\infty \Sigma^d\Omega^k\fib(\Sigma^\infty g)\simeq ev_{d-k}\fib(\Sigma^\infty g)
  \] 
  because $\Sigma$ and $\Omega$ are homotopy inverses in $\Sp(\cC)$. Now notice that by the description of the functor $\Sigma^\infty$ given in Lemma \ref{lem:sigmainfinity} and Remark \ref{rmk:pointedstabilization}, we can write $ev_{d-k}\Sigma^\infty g$ as the colimit of the morphisms
  \[
    \Omega^n f_C^+(S^{d-k+n})\simeq \Omega^n\Sigma^{d-k+n}C\to\Omega^n\Sigma^{d-k+n}D\simeq \Omega^n f_D^+(S^{d-k+n})
  \] where $\Omega$ is being applied in $\cC$ rather than $\Sp(\cC)$. We have the following string of equivalences:
\begin{align*}
&\fib\left( \colim_n\left(\Omega^n\Sigma^{d-k+n}C\to\Omega^n\Sigma^{d-k+n}D\right)\right)\\
&\simeq \colim_n\left( \fib\left(\Omega^n\Sigma^{d-k+n}C\to\Omega^n\Sigma^{d-k+n}D\right)\right)\\
&\simeq \colim_n\left(\Omega^n \fib\left(\Sigma^{d-k+n}C\to\Sigma^{d-k+n}D\right)\right)
\end{align*}
in which the first equivalence uses the fact that finite limits commute with sequential colimits in a differentiable category. By Corollary~\ref{cor:freudenthal} we have that $\Sigma^{d-k+n}C\to\Sigma^{d-k+n}D$ is $(d+n)$-connected, so by Lemma \ref{lem:connectedfiberinpointed} we have that $\fib\left(\Sigma^{d-k+n}C\to\Sigma^{d-k+n}D\right)$ is $(d+n-1)$-connected. The result follows from Proposition \ref{prop:pullbackconnectivity} and the fact that the connectivity structure of $\cC$ is compatible with sequential colimits.
\end{proof}

In the unpointed case, Corollary~\ref{cor:amalgconnectivity} allows us to compose with a disjoint basepoint functor $X \mapsto X \amalg \ast$.

\begin{cor}
  Suppose that $\cC$ is differentiable, admits pushouts, and has an initial object. Suppose that the connectivity structure on $\cC$ is compatible with cobase change and sequential colimits. Then the stabilization functor $\Sigma^\infty_+\co \cC \to \Sp(\cC)$ preserves connectivity.
\end{cor}

Finally, in the relative case, we have the following assembly of the above results.

\begin{thm}
  \label{thm:stabilizationconnectivity}
  Suppose that $\cC$ is differentiable, admits pushouts, and has an initial object. Suppose that the connectivity structure on $\cC$ is compatible with cobase change and sequential colimits. Then, for any object $Z \in \cC$, the stabilization functor
  \[
    \Sigma^\infty_{Z,+}\co \cC_{/Z} \to \Sp(\cC_{/Z})
  \]
  and its pointed variant $\Sigma^\infty_Z$ preserve connectivity.
\end{thm}

\subsection{Cartesian squares}

\begin{defn}\label{def:cartesian}
  A commutative square
  \[
    \xymatrix{X \ar[r] \ar[d] & Y \ar[d] \\ Z \ar[r] & W}
  \]
  in $\cC$ is $k$-Cartesian if the induced map $X \to Y \times_W Z$ to the homotopy pullback is $k$-connected (cf. \cite[1.3]{goodwillie-calcii}).\footnote{More generally, just as in the classical case we have notions of $k$-Cartesian cubes, and $k$-coCartesian cubes, in terms of connectivity of maps to total homotopy pullbacks and from total homotopy pushouts}
\end{defn}

As expected, there is a ``stacking lemma'' for $k$-Cartesian squares.

\begin{prop}
  \label{prop:cartesianstacking}
  Given a commutative diagram
  \[
    \xymatrix{
      X \ar[r] \ar[d] & Y \ar[r] \ar[d] & Z \ar[d] \\
      X' \ar[r] & Y' \ar[r] & Z'
    }
  \]
  with both squares $k$-Cartesian, the outside rectangular diagram is also $k$-Cartesian.
\end{prop}

\begin{proof}
  The map $X \to Z \times_{Z'} X'$ is the composite
  \[
    X \to Y \times_{Y'} X' \to (Z \times_{Z'} Y') \times_{Y'} X'
    \simeq Z \times_{Z'} X'.
  \]
  The first map is $k$-connected, and the second is the pullback of a $k$-connected map.
\end{proof}

\begin{rmk}
  We say that \emph{pullbacks reflect connectivity} in $\cC$ if the connectivity of any pullback of $f$ is the same as the connectivity of $f$. For example, this is true of stable connectivity structures by Proposition~\ref{prop:stablecofiber}.

  When pullbacks reflect connectivity, the proof of Proposition~\ref{prop:cartesianstacking} can be upgraded to show that Cartesianness determines a connectivity structure on the arrow category of $\cC$.
\end{rmk}

\section{Skeleta}
\label{sec:skeleta}

\subsection{Skeletal objects}

\begin{defn}
  \label{def:skeletalobject}
  An object $B$ is \emph{$k$-skeletal} if, for any $k$-connected map $X \to Y$, every map $B \to Y$ in $h\cC$ lifts to a map $B \to X$.

  A map $B\to Y$ is \emph{relatively $k$-skeletal} if it is $k$-skeletal when viewed as a map in the undercategory $\cC_{A/}$.
\end{defn}

\begin{rmk}
  If $\cC$ has an initial object $\varnothing$, then an object $B$ is $k$-skeletal if and only if the map $\varnothing \to B$ is relatively $k$-skeletal.
\end{rmk}

\begin{rmk}
Relative $k$-skeletality can be phrased as a lifting property: given a commutative diagram in $\cC$
  \[
  \xymatrix{
  A\ar[r]^{\alpha}\ar[d]_{\beta} & X\ar[d]^{\gamma}\\
  B\ar[r]_{\delta} & Y
  }
  \]
  where $\gamma$ is $k$-connected, we ask for the existence of a lift $B \to X$ together with coherences. This asks that the map to the homotopy pullback
  \[
    \Map(B,X) \to \Map(B,Y) \times_{\Map(A,Y)} \Map(A,X)
  \]
  is surjective on $\pi_0$.
  
  More explicitly, we can consider $\delta$ to be a point of the mapping space $\cC_{A/}(\beta,\gamma\alpha)$. Note that, given some $\phi\in\cC_{A/}(\beta,\alpha)$, we can compose with $\gamma$ to get a point $\gamma\phi\in\cC_{A/}(\beta,\gamma\alpha)$, i.e.~there is a map of spaces $\gamma_\ast\co\cC_{/A}(\beta,\alpha)\to \cC_{A/}(\beta,\gamma\alpha)$. By \cite[5.5.5.12]{lurie-htt} the fiber over $\delta$ of this map is precisely the space of maps from $\alpha$ to $\beta$ in $(\cC_{A/})_{/\delta}$. Therefore relative $k$-skeletality can be rephrased as asking for $\Lift(\beta,\alpha)$, the mapping space in $(\cC_{A/})_{/\delta}$, to be non-empty. Or, equivalently, for the morphism $\gamma_\ast$ to be be surjective on $\pi_0$. This should be compared to the discussion regarding \textit{unique} lifts in \cite[5.2.8.3]{lurie-htt}.
\end{rmk}

\begin{exam}
	Recall that a map $X\to Y$ in $\Spaces$ is called $k$-connected if it is an isomorphism on $\pi_i$ for all $i<k$ and a surjection on $\pi_k$. The cases of $i=0$ and $i=k$ in particular imply that $S^k$ is $k$-skeletal in $\Spaces$, in the sense of Definition \ref{def:skeletalobject}.
\end{exam}

\begin{prop}
  If $A$ is $k$-skeletal and $X \to Y$ is $(k+d)$-connected for $d \geq 0$, then the map of function spaces $\Map(A,X) \to \Map(A,Y)$ is $d$-connected.
\end{prop}

\begin{proof}
  Recall the following inductive characterization of connectedness in $\Spaces$: a map $U \to V$ is $0$-connected if it is surjective on $\pi_0$, and it is $d$-connected for $d > 0$ if it is $0$-connected and the map $U \to \holim(U \to V \from U)$ to the homotopy fiber product is $(d-1)$-connected (cf.~for instance \cite[6.5.1.18]{lurie-htt}).

  We now prove the proposition by induction on $d$. Asking that the map $\Map(A,X) \to \Map(A,Y)$ be $0$-connected is the same as asking that every map $A \to Y$ in $h\cC$ lifts to $X$, which is implied by the definition of $k$-skeletality.

  Therefore, to complete the induction we need to show that if $X \to Y$ is $(k+d)$-connected, the diagonal map
  \[
    \Map(A,X) \to \holim(\Map(A,X) \to \Map(A,Y) \from \Map(A,X))
  \]
  is $(d-1)$-connected. However, the functor $\Map(A,-)$ preserves homotopy limits; therefore, this is equivalent to asking that the map
  \[
    \Map(A,X) \to \Map(A, \holim (X \to Y \from X)),
  \]
  induced by the diagonal $\Delta\co X \to \holim(X \to Y \from X)$, is $(d-1)$-connected. However, by Corollary~\ref{cor:diagonalconnectivity} the map $\Delta$ is $k+(d-1)$-connected, and so the inductive hypothesis completes the proof.
\end{proof}

\begin{cor}
  \label{cor:uniquelifts}
  If $A$ is $k$-skeletal and $f\co X \to Y$ is $(k+1)$-connected, then
  every map $A \to Y$ in $h\cC$ lifts uniquely to a map $A \to X$ in $h\cC$.
\end{cor}

\begin{proof}
  The map of spaces $\Map(A,X) \to \Map(A,Y)$ is an isomorphism on $\pi_0$, which equivalently says that $\Hom_{h\cC}(A,X) \to \Hom_{h\cC}(A,Y)$ is an isomorphism.
\end{proof}

\begin{defn}
  \label{def:liftcategory}
  Let $\phi\co X\to Y$ and $f\co A\to Y$ be maps in $\cC$. Then we define $\Lift(f,\phi)$ to be the following homotopy pullback:
  \[
    \xymatrix{
      \Lift(f,\phi) \ar[r] \ar[d] & \Map(A,X) \ar[d]^-{\phi_\ast} \\
      \ast \ar[r]_-{f} & \Map(A,Y).
    }\]
\end{defn}

\begin{rmk}\label{rmk:liftashom}
  By \cite[5.5.5.12]{lurie-htt} the space $\Lift(f,\phi)$ is precisely the space of morphisms from $f$ to $\phi$ in $\cC_{/Y}$.
\end{rmk}

\begin{prop}
  \label{prop:adjointskeletal}
  Let $L\co \cC \to \dD$ be a functor with a right adjoint $R$. If $R$ preserves $k$-connected morphisms then $L$ preserves $k$-skeletal objects.
\end{prop}

\begin{proof}
  Let $A$ be $k$-skeletal in $\cC$ and suppose we have a diagram
  \[
    \xymatrix{
      & X\ar[d]^f\\
      L(A)\ar[r]_{\phi}& Y}
  \]
  in $\dD$. Then by adjunction we have a diagram
  \[
    \xymatrix{
      & R(X)\ar[d]^{R(f)}\\
      A\ar[r]_-{adj(\phi)}& R(Y)}
  \] 
  in $\cC$ and by the hypothesis, $R(f)$ is $k$-connected. Therefore there is a lift
  \[
    \xymatrix{
      & R(X)\ar[d]^{R(f)}\\
      A\ar[r]_-{adj(\phi)}\ar[ur]^{\wt{adj(\phi)}}& R(Y)}
  \] 
  by the $k$-skeletality of $A$. Applying $L$ to above diagram and using the counit natural transformation of the adjunction, we obtain the following commutative diagram in $h\dD$:
  \[
    \xymatrix{
      & LR(X)\ar[d]^{LR(f)}\ar[r]^-{\epsilon_X}& X\ar[d]^f\\
      L(A)\ar[r]_-{L(adj(\phi))}\ar[ur]^{L(\wt{adj(\phi)})}& LR(Y)\ar[r]_-{\epsilon_Y}&Y}
  \]
  However, the lower composite is $\phi$.
\end{proof}

\begin{cor}\label{cor:adjointrelativeskeletal}
  The left adjoint $L$ preserves relatively $k$-skeletal morphisms.
\end{cor}

\begin{proof}
  The adjunction between $L$ and $R$ induces an adjunction between the slice categories $\cC_{A/}$ and $\dD_{L(A)/}$ by \cite[5.2.5.1]{lurie-htt}, and we can apply the proposition to the slice categories.
\end{proof}

\subsection{Colimits of skeletal objects}

A standard result about left-lifting properties implies the following.
\begin{prop}
  \label{prop:skeletalsaturation}
  Relatively $k$-skeletal maps are closed under equivalences, coproducts, composition, pushout, and retracts.
\end{prop}

\begin{prop}
  \label{prop:coproductskeletal}
  A coproduct of objects is $k$-skeletal if all the summands are $k$-skeletal.
\end{prop}

\begin{proof}
  Let $\{A_j\}_{j \in J}$ be a collection of $k$-skeletal objects, and suppose $\phi\co X\to Y$ is a $k$-connective morphism.  Let $i_j\co A_j\to \coprod A_j$ be the standard inclusions and let $f\co \coprod A_j \to Y$ be any map so that $\coprod f\circ i_j\simeq f$ in $\cC_{/Y}$.  By Remark \ref{rmk:liftashom} we have that
  \[
    \Lift(f, \phi)\simeq\Lift\left(\amalg_j f\circ i_j,\phi\right)\simeq \prod_j \Lift(f\circ i_j,\phi).
  \]
  Because all $A_j$ are $k$-skeletal, the right-hand side is nonempty; therefore, the left-hand side is nonempty.
\end{proof}

\begin{rmk}
  This result can be upgraded to an if-and-only-if statement if \emph{enough maps exist}: if there exist maps $A_j \to Y$ for every object $Y$. For example, this is true if $\cC$ is pointed.
\end{rmk}

\begin{prop}
  \label{prop:pushoutskeletal}
  Suppose that we have a homotopy pushout diagram
  \[
    \xymatrix{
      A \ar[r]^p \ar[d]_q & B \ar[d]^{i_B} \\
      C \ar[r]_{i_C} & P.
    }
  \]
  If $B$ and $C$ are $k$-skeletal and $A$ is $(k-1)$-skeletal, then the homotopy pushout $P$ is $k$-skeletal.
\end{prop}

\begin{proof}
  Let $f\co X\to Y$ be a $k$-connected map, $\phi\co P\to Y$ any map, and consider the following diagram:

\[
\xymatrix{
& B\ar[dr]^{i_B}& & X\ar[d]^f\\
A\ar[ur]^{p}\ar[dr]_q & & P\ar[r]^\phi & Y\\
& C\ar[ur]_{i_C} & &
}
\] 

Because $B$ and $C$ are $k$-skeletal, there are lifts $\wt{\phi i_B}\co B\to X$ and $\wt{\phi i_C}\co C\to X$.
We can obtain lifts of $\phi i_B p$ and $\phi i_C q$ to $X$ by simply composing $p$ and $q$ with $\wt{\phi i_B}$ and $\wt{\phi i_C}$ respectively. Using the fact that $f$ is not just $(k-1)$-connected, but $k$-connected, and Proposition \ref{cor:uniquelifts}, both of these lifts must agree in $h\cC$ because they all lift $\phi i_B p\simeq \phi i_C q$. Call that unique lift $\psi\co A\to X$. Since $\psi \simeq \wt{\phi i_B}p \simeq \wt{\phi i_C} q$, the universal property of $P$ induces a lift $\wt\phi\co P\to X$.
\end{proof}

\subsection{Skeleta and minimal skeleta}

\begin{defn}
  \label{def:skeleton}
  A map $f\co A \to X$ is a \emph{$k$-skeleton} if $A$ is $k$-skeletal and $f$ is $k$-connected. If $f$ is $(k+1)$-connected, we refer to this as a \emph{minimal $k$-skeleton}.

  Similarly, a factorization $A \to B \to X$ is a \emph{relative $k$-skeleton} if it is a $k$-skeleton in the undercategory $\cC_{A/}$.
\end{defn}

\begin{rmk}
Notice that saying a factorization $A\to B\to X$ of a map $A\to X$ is a relative $k$-skeleton is equivalent to saying that $A\to B$ is relatively $k$-skeletal and that $B\to X$ is $k$-connected. We will show in Section \ref{sec:cellularskeleta} that in good cases, every map of $\cC$ admits such a factorization. In other words, the classes of relatively $k$-skeletal maps and $k$-connected maps are the left and right classes of a weak factorization system, respectively.
\end{rmk}

\begin{exam}
  Every $k$-skeletal object is a minimal $k$-skeleton of itself.
\end{exam}

\begin{exam}
  Suppose that $A$ is a CW-complex with $k$-skeleton $A^{(k)}$. Then these spaces are built, inductively, by a sequence of (homotopy) pushout diagrams:
  \[
    \xymatrix{
      \coprod S^{k-1} \ar[r] \ar[d] & A^{(k-1)} \ar[d] \\
      \coprod D^k \ar[r] & A^{(k)}
    }
  \]
  By Proposition~\ref{prop:coproductskeletal}, the space $\coprod S^{k-1}$ is $(k-1)$-skeletal and the space $\coprod D^k$ is $0$-skeletal. Inductively assuming that $A^{(k-1)}$ is $(k-1)$-skeletal, then by Proposition~\ref{prop:pushoutskeletal} we find that $A^{(k)}$ is $k$-skeletal.

  The maps $A^{(k)} \to A$ are then $k$-skeleta. More generally, if $X$ is a space and $A \to X$ is a weak equivalence from a CW-complex, then the maps $A^{(k)} \to X$ are $k$-skeleta.
\end{exam}

The following proposition shows that a minimal skeleton is a retract of any other skeleton.

\begin{prop}
  \label{prop:minimalretract}
  Suppose that $A \to X$ is a minimal $k$-skeleton and $B \to X$ is another $k$-skeleton. Then there exists a map $r\co B \to A$ of objects over $X$, unique in the homotopy category, and it admits a section.
\end{prop}

\begin{proof}
  Because the map $A \to X$ is $(k+1)$-connected by assumption, Corollary~\ref{cor:uniquelifts} implies that in the homotopy category $h\cC$ there is a unique lift $r\co B \to A$ over $X$. To show the existence of a section, the definitions imply that there is a lift $i\co A \to B$ over $X$.   The maps $ri$ and $\id_A$ are then two lifts $A \to A$ along a $(k+1)$-connected map $A \to X$, and another application of Corollary~\ref{cor:uniquelifts} implies that $ri = \id_A$ in $h\cC$.
\end{proof}

If we have two minimal $k$-skeleta of $X$, we can then arrive
at the following conclusion.

\begin{prop}
  \label{prop:skeletaluniqueness}
  Suppose that $A_1 \to X$ and $A_2 \to X$ are minimal $k$-skeleta of $X$. Then in the homotopy category $h\cC$ there exists a unique isomorphism $A_1 \simeq A_2$ over $X$.
\end{prop}

\begin{cor}\label{cor:skeletalequivalence}
  Any $(k+1)$-connected map between $k$-skeletal objects is an equivalence.
\end{cor}

\begin{exam}\label{example: retracts of kdim CW}
  Retracts of $k$-dimensional CW-complexes are $k$-skeletal. For the converse, suppose $X$ is a $k$-skeletal space, and construct a CW-approximation $A \to X$. Then $A^{(k)} \to X$ is a $k$-skeleton and $X \to X$ is a minimal $k$-skeleton, so $X$ is a retract of a $k$-dimensional CW-complex $A^{(k)}$. 
  
  In particular, a connected $1$-skeletal space $X$ is a retract of a wedge of circles. This means $X$ is a $K(G,1)$ for $G$ a retract of a free group; such a group $G$ is free, and choosing generators gives an equivalence from a wedge of circles to $X$.
\end{exam}

\subsection{Cells}

\begin{defn}
  \label{def:cell}
  A \emph{$k$-cell} is a $(k-1)$-connected map $A \to B$ such that $A$ is $(k-1)$-skeletal and the map is relatively $k$-skeletal.

  We say that a set $S_k$ of $k$-cells is a set of \emph{generating $k$-cells} for the connectivity structure if it satisfies the following property: a map is $k$-connected if and only if it is $(k-1)$-connected and satisfies the right lifting property with respect to the maps in $S_k$.

  A set of \emph{generating cells} is a choice of sets $S_k$ of generating $k$-cells for all $k$; if one exists, we say that the connectivity structure is \emph{determined by cells}.
\end{defn}

\begin{defn}\label{def:boundedcell}
  A cell $A \to B$ is \emph{$j$-bounded} if $B$ is $j$-skeletal.
\end{defn}

\begin{exam}
  For $k \geq 0$ the inclusion $S^{k-1} \to D^k$ is a ($0$-bounded) $k$-cell in $\Spaces$, and the connectivity structure on spaces is determined by these cells.
\end{exam}

\begin{exam}
  Similarly, for $k \in \mb Z$ the maps $\Sigma^{k-1} \mb S \to \ast$ and $\ast \to \Sigma^k \mb S$ are both $k$-cells ($(-\infty)$-bounded and $k$-bounded respectively) for the standard connectivity structure on the category $\Sp$ of spectra. The connectivity structure on spectra is determined by the first type of cells: a map $f\co X \to Y$ has the right lifting property with respect to $\Sigma^{k-1} \mb S \to \ast$ if and only if $\pi_k \cofib(f)$ is trivial.

  By contrast, the connectivity structure is not determined by the latter cells. A $(k-1)$-connected map $f\co X \to Y$ has the right lifting property with respect to $\ast \to \Sigma^k \mb S$ if and only if $\pi_k(X)$ surjects onto $\pi_k(Y)$, which does not ensure that the map $\pi_{k-1}(X) \to \pi_{k-1}(Y)$ is an isomorphism.
\end{exam}

\begin{exam}
  Consider the category $\Sp_{\geq 0}$ of connective spectra. The connectivity structure on $\Sp$ restricts to one on $\Sp_{\geq 0}$.\footnote{The inclusion $\Sp_{\geq 0} \to \Sp$ does not preserve homotopy pullbacks, and so the connectivity structure on $\Sp_{\geq 0}$ is not lifted from $\Sp$. Instead, it is lifted from $\Spaces$ along the functor $\Omega^\infty$.} The maps $\Sigma^{k-1} \mb S \to \ast$ for $k \geq 1$ are also $k$-cells in $\Sp_{\geq 0}$, but they do not determine the connectivity structure: they do not detect surjectivity on $\pi_0$. To repair this, we must also include the $0$-cell $\ast \to \mb S$.

  Alternatively, the maps $\Sigma^\infty_+ S^{k-1} \to \Sigma^\infty_+ D^k$ and $\ast \to \Sigma^\infty_+ D^0$ are also generating cells for $\Sp_{\geq 0}$.
\end{exam}

\begin{prop}
  \label{prop:adjointspreservecells}
  Suppose that $\dD$ has a connectivity structure lifted from $\cC$ along a functor $R\co \dD \to \cC$ with left adjoint $L$.

  If $L$ preserves connectivity, then $L$ takes $k$-cells to $k$-cells. If, additionally, there are sets $S_k$ of generating $k$-cells that determine the connectivity structure on $\cC$, then the sets $L(S_k)$ are generating $k$-cells that determine the connectivity structure on $\dD$.
\end{prop}

\begin{proof}
  Suppose $A \to B$ is a $k$-cell in $\cC$. Then $L(A)$ is $(k-1)$-skeletal by Proposition~\ref{prop:adjointskeletal} and $L(A) \to L(B)$ is relatively $k$-skeletal by Corollary~\ref{cor:adjointrelativeskeletal}. Therefore, if $L$ preserves connectivity then $L(A)\to L(B)$ is a $k$-cell.   
  
  Suppose now that the connectivity structure on $\cC$ is determined by cells. A map $f$ in $\dD$ is $n$-connected if and only if $Rf$ is $n$-connected, and this is true if and only if $Rf$ has the right lifting property with respect to the generating cells $A \to B$ in $S_k$ for $k \leq n$; by adjunction, this is true if and only if $f$ has the right lifting property with respect to the cells $LA \to LB$.
\end{proof}

\begin{exam}
  If $S$ is an associative ring spectrum, the category $\LMod_S$ of left $S$-modules has a connectivity structure lifted from the forgetful functor $R\co \LMod_S \to \Sp$. There is a left adjoint $L(X) = S \otimes X$, and it preserves connectivity if and only if $S$ is connective. In this case, the maps $\Sigma^{k-1} S \to \ast$ are $k$-cells that determine the connectivity structure on $\LMod_S$. Similarly, $\Sigma^{k-1} S \to \ast$ for $k \geq 1$ and $\ast \to S$ are cells that determine the connectivity structure on $(\LMod_S)_{\geq 0}$.
\end{exam}

\begin{exam}\label{exam:connectivityforcommutativealgebras}
  Consider the category $\CAlg(\Sp)$ of commutative algebras in $\Sp$, with the connectivity structure lifted from $\Sp$. The left adjoint to the forgetful functor is the free algebra functor $\mb P$:
  \[
    \mb P(X) \simeq \coprod_{k \geq 0} X^{\otimes k}_{h\Sigma_k}.
  \]
  The functor $\mb P$ does not preserve connectivity in general, because the symmetric power functors don't: the map $S^{-1} \to \ast$ is $(-1)$-connected, but the map $(S^{-1} \otimes S^{-1})_{h\Sigma_2} \to \ast$ is only $(-2)$-connected. However, if we restrict attention to the category $\Sp_{\geq 0}$ of connective spectra, the free functor $\mb P$ does preserve connectivity.   Therefore, the maps $\mb P(\Sigma^{k-1} \mb S) \to \mb S$ and $\mb S \to \mb P(S)$ are generating cells for the connectivity structure on $\CAlg(\Sp_{\geq 0}$) lifted from $\Sp_{\geq 0}$.

  (A similar result applies to an arbitrary operad acting on the category $(\LMod_S)_{\geq 0}$ of connective modules over a connective commutative ring spectrum.)
\end{exam}

\begin{exam}\label{exam:connectivityforalgebras}
  Similarly, if $R$ is a connective $\mb E_2$-ring spectrum then there is a connectivity structure on $\Alg(\LMod_R)_{\geq 0}$ lifted along the forgetful functor $\Alg(\LMod_R)_{\geq 0}\to (\LMod_R)_{\geq 0}$. This connectivity structure is determined by the cells $\mb T^R(\Sigma^{k-1} R)\to R$ for $k\geq 1$ and $R\to \mb T^R(R)$, where $\mb T^R\co \LMod_R \to \Alg(\LMod_R)$ is the free $R$-algebra functor.

  (A similar result applies to $\mb E_k$-algebras in the category of left modules over a connective $\mb E_{k+1}$-ring spectrum.)
\end{exam}

\subsection{Cellular skeleta}\label{sec:cellularskeleta}

\begin{defn}
  We say that a set $S_k$ of $k$-cells \emph{is sufficient for $k$-skeleta} if the following properties hold.
  \begin{enumerate}
  \item The set $S_k$ is a set of generating $k$-cells for $\cC$.
  \item $\cC$ has homotopy colimits.
  \item Given any cells $A_i \to B_i$ in $S_k$, any cobase change of the map $\coprod A_i \to \coprod B_i$ is $(k-1)$-connected.
  \end{enumerate}
\end{defn}

\begin{rmk}
  In particular, the third item is automatic if the connectivity structure is compatible with cobase change.
\end{rmk}

\begin{prop}
  \label{prop:skeletaexist}
  Suppose that, for all $k > N$, we have a set of $k$-cells $S_k$ that is sufficient for $k$-skeleta. Then an $N$-connected map $f\co A \to X$ admits relative $n$-skeleta for any $n$.
\end{prop}

\begin{proof}
  The factorization $A \to A \too{f} X$ is a relative $n$-skeleton for any $n \leq N$: the map $A \to A$ is relatively $n$-skeletal because it is initial in the undercategory $\cC_{/A}$, and the map $f$ is $N$-connected by assumption.

  We then apply a standard cellular approximation method. Suppose inductively that we have found a relative $(n-1)$-skeleton $A \to X^{(n-1)} \to X$. Let $S_n = \{A_j \to B_j\}_{j \in J}$ be a set of generating $n$-cells for the connectivity structure, and consider the collection $T$ of all (up to equivalance) commutative diagrams
  \[
    \xymatrix{
      A_j \ar[r] \ar[d] & X^{(n-1)} \ar[d] \\
      B_j \ar[r] & X.
    }
  \]
  Factor the map $X^{(n-1)} \to X$ through the homotopy pushout
  \[
    \xymatrix{
      \coprod_{t \in T} A_{j_t} \ar[r] \ar[d] & X^{(n-1)} \ar[d] \\
      \coprod_{t \in T} B_{j_t} \ar[r] & X^{(n)}.
    }
  \]
  By assumption, this cobase change $X^{(n-1)} \to X^{(n)}$ is $(n-1)$-connected. Therefore, the factorization $X^{(n-1)} \to X^{(n)} \to X$ shows that the map $X^{(n)} \to X$ is at least $(n-1)$-connected using Definition~\ref{def:connectivity}. Moreover, since $\coprod A_j \to \coprod B_j$ is relatively $n$-skeletal, so is the map $X^{(n-1)} \to X^{(n)}$ and hence the composite $A \to X^{(n-1)} \to X^{(n)}$ by Proposition~\ref{prop:skeletalsaturation}.

  To prove that $X^{(n)} \to X$ is $n$-connected, it then suffices to show that any diagram
  \[
    \xymatrix{
      A_j \ar[r] \ar[d] & X^{(n)} \ar[d] \\
      B_j \ar[r] & X.
    }
  \]
  has a lift when the left-hand vertical arrow is a cell in the generating set $S_n$. However, in that case $A_j$ is $(n-1)$-skeletal, and so the topmost map factors through a map $A_j \to X^{(n-1)}$; the resulting commutative diagram is then equivalent to a map in $T$, and so there is a lift $B_j \to X^{(n)}$ by construction of the homotopy pushout.
\end{proof}

\begin{rmk}
  We refer to this type of construction of a factorization $A \to B \to X$, via iterated pushout, as a \emph{cellular} construction of a map using the cells in $S_k$---regardless of whether or not $B$ is a relative skeleton of $X$.
\end{rmk}

\begin{rmk}
Suppose that the connectivity structure of $\cC$ is determined by sets of $k$-cells $S_k$ and all of the $S_k$ are sufficient for $k$-skeleta. Then the above proposition implies that for every $k$  there is a weak factorization system on $h\cC$ whose left class is the class of relative $k$-skeleta and whose right class is the class of $k$-connected morphisms.
\end{rmk}

\begin{cor}
  If $\varnothing$ is an initial object and $\varnothing \to X$ is $N$-connected, then $X$ admits $n$-skeleta for any $n$.
\end{cor}

\begin{exam}
  All objects in the category $\Spaces$ of spaces have skeleta, and the proof above is a standard construction of a CW-approximation to any space using the cellular approximation theorem.
\end{exam}

\begin{exam}
  Suppose $S$ is a connective ring spectrum. The forgetful functor $\LMod_S \to \Sp$ detects homotopy colimits and connectivity, and so $\LMod_S$ satisfies the assumptions of the proposition. Therefore, an object in $\LMod_S$ has an $n$-skeleton so long as it is $N$-connected for some $N > -\infty$.
\end{exam}

\section{Detecting skeleta}
\label{sec:detectingskeleta}

\subsection{Nilpotence}

\begin{defn}
  \label{def:nilpotentmaps}
  Suppose $E$ is a collection of maps in $\cC$. The class of \emph{$E$-nilpotent maps} is the smallest collection of maps in $h\cC$ which contains $E$ and is closed under base change, transfinite composition, products, filtered homotopy limits, and retracts.
\end{defn}

\begin{exam}\label{exam:classicalnilpotence}
  Let $E$ be the collection of maps $\ast \to K(A,n)$ in $\Spaces$ for $n \geq 2$. Recall that if a space $X$ is nilpotent in the classical sense (i.e.~$X$ has nilpotent fundamental group, which acts nilpotently on the higher homotopy groups) then it can be written as a limit of a tower $X\simeq\lim(\cdots\to Y_n\to Y_{n-1}\to\cdots Y_2\to Y_1\to Y_0\simeq \ast)$ in which $Y_{n}\to Y_{n-1}$ is the fiber of a map $Y_{n-1}\to K(A,m)$ for some $A$ and some $m\geq 2$ (cf. \cite[3.2.2]{mayponto-moreconcise}).  Then it follows from Definition \ref{def:nilpotentmaps} that the terminal morphism $X \to \ast$ is $E$-nilpotent.
\end{exam}

\begin{prop}
  An object $A$, or a map $A \to B$, has the left lifting property with respect to $E$-nilpotent maps if and only it has the left lifting property with respect to maps in $E$.
\end{prop}

\begin{proof}
  This is the assertion that ``right lifting properties'' are closed under the operations that define $E$-nilpotent maps.
\end{proof}

\begin{cor}
  \label{cor:nilpotentadjunctions}
  Suppose $R\co \dD \to \cC$ is a functor with a left adjoint $L$ and $E$ is a collection of maps in $\dD$. An object $A$ of $\cC$ has the lifting property with respect to $R(E)$-nilpotent maps if and only if $LA$ has the lifting property with respect to maps in $E$.
\end{cor}

\begin{exam}
  \label{exam:pi0skeletal}
  Suppose that $S$ is a connective ring spectrum, and let $E$ be the class of maps $\ast \to \Sigma^n HN$ for $N$ a $\pi_0(S)$-module and $n \geq k+1$. Then any $k$-connected map $X \to Y$ in $\LMod_S$ is $E$-nilpotent: one can apply the Postnikov tower to its cofiber. Therefore, if a left $S$-module $M$ has a $k$-skeletal image
  \[
    H\pi_0(S) \otimes_S M
  \]
  in the category $\LMod_{H\pi_0(S)}$, then $M$ is $k$-skeletal. The converse holds by Proposition~\ref{prop:adjointskeletal}.
\end{exam}

\subsection{Tangent categories}

We begin with a reminder about tangent $\infty$-categories.

Suppose that $\cC$ is a presentable $\infty$-category. Let $T_\cC \to \Fun(\Delta^1, \cC)$ be a \emph{tangent bundle} to $\cC$ in the sense of \cite[7.3.1.9, 7.3.1.10]{lurie-higheralgebra}, i.e.~the category of excisive functors $\Exc(\sfin,\cC)$. There are functors
\[
  ev_d\co T_{\cC}\to \cC
\]
for $d\geq 0$ given by evaluation at $S^d$,
\[
  U\co T_{\cC}\to \cC
\]
given by evaluation at $\ast$, and a map
\[
  T_{\cC}\to \Fun(\Delta^1,\cC)
\]
that takes $X$ to $X(S^0)\to X(\ast)$.  By restricting to functors which take $\ast$ to a fixed object $Z\in C$ we obtain \emph{pointed} excisive functors $\Exc_\ast(\sfin,\cC_{/Z})$. In other words, the fiber over an object $Z\in \cC$, along the map $U\co T_{\cC}\to \cC$, is equivalent to the category $\Sp(\cC_{/Z})$ discussed in Section~\ref{sec:stabilization}.

The results of \cite[7.3.1]{lurie-higheralgebra} imply that for any morphism $f\co W\to Z$ in $\cC$ there is an pair of functors  $f^\ast\co \Sp(\cC_{/Z})\to \Sp(\cC_{/W})$ and $f_!\co \Sp(\cC_{/W})\to \Sp(\cC_{/Z})$ that form an adjunction $f_!\dashv f^\ast$. The functor $f_*$ is described as follows: given an excisive functor $Y\co \sfin\to \cC$ with $Y(\ast) = Z$,
\[
  (f^\ast Y)(K) = Y(K) \times_Z W.
\]
Its left adjoint $f_!$ is described as follows: given an excisive functor $X\co \sfin \to \cC$ with $X(\ast) = W$, the image $f_! X$ is the excisive approximation to the functor sending $K$ to the pushout $X(K) \amalg_W Z$.

Up to equivalence of categories, the objects of $T_\cC$ are pairs $(Z,M)$ of an object $Z$ of $\cC$ and an object $M$ of the stable category $\Sp(\cC_{/Z})$. A morphism from $(W,M)$ to $(Z,N)$ in $T_{\cC}$ is equivalent to a pair of morphisms $(f,\phi)$ where $f\co W\to Z$ is a morphism in $\cC$ and $\phi\co M\to f^\ast(N)$ is a morphism in $\Sp(\cC_{/W})$ (or, equivalently, a morphism $\phi\co f_!(M)\to N$ in $\Sp(\cC_{/Z})$). Under this identification, the functor $T_{\cC}\to\Fun(\Delta^1, \cC)$ sends the object $(Z,M)$ to the object $\Omega^\infty M \to Z$ in $\cC_{/Z}\subseteq \Fun(\Delta^1,\cC)$.

\begin{prop}\label{prop:tangentcatconnectivity}
  Suppose $\cC$ has a connectivity structure. Then there is a \emph{fiberwise connectivity structure} on $T_\cC$, defined as follows. A map $X \to Y$ in $T_{\cC}$ is $k$-connected if the map $UX \to UY$ is an equivalence and each map $ev_n X \to ev_n Y$ is $(k-n)$-connected for $n\geq 0$.
\end{prop}

\begin{proof}
  It is clear that $k$-connected maps are $(k-1)$-connected and that equivalences are $\infty$-connected. Moreover, connectivity is preserved by pullback because pullbacks in the category $T_\cC$ are pullbacks of functors to $\cC$.

  Suppose that we have composable maps $f\co X \to Y$ and $g\co Y \to Z$ in $T_\cC$, and two out of three of $f$, $g$, and $gf$ are $k$-connected for some $k$. Then the maps $UX \to UY \to UZ$ are all equivalences. Therefore, this diagram is equivalent to a diagram in the full subcategory $\Sp(\cC_{/UX})$. The restriction of connectivity to this subcategory is the stable connectivity structure of Definition~\ref{def:stablestructure}, where the 2-out-of-3 axioms for connectivity are satisfied.
\end{proof}

\begin{rmk}\label{rmk:loopsinfinitypreservesconnectivity}
Note that the functors $ev_0$ and $U$ both preserve connectivity; the former by definition and the latter because equivalences are $k$-connected for all $k$.
\end{rmk}

\begin{rmk}
When restricted to the fiber over $Z\in\cC$ the connectivity structure of Proposition~\ref{prop:tangentcatconnectivity} recovers the stable connectivity structure on $\Sp(\cC_{/Z})$ induced by stabilizing the connectivity structure on $\cC_{/Z}$ (which is in turn induced by the connectivity structure on $\cC$). It is not hard to check that, for a morphism $f\co Z\to W$ in $\cC$, the pullback functor $f^\ast$ preserves connectivity. The pushforward $f_!$ preserves connectivity if the connectivity structure is compatible with cobase change and $\cC$ is differentiable.
\end{rmk}

\begin{rmk}
  The base-change functors $f^*$ induced by maps $f\co W \to Z$ give the following interpretation of the fiberwise connectivity structure. The map $T_\cC \to \cC$ is a Cartesian fibration, whose fiber over $Z$ is the stable category $\Sp(\cC_{/Z})$. Each fiber has the associated stable connectivity structure from Definition~\ref{def:liftedstructure}, and a map $f$ in $\cC$ induces a connectivity-preserving functor $f^*$. This fibration determines, by straightening, a contravariant functor from $\cC$ to a category of stable categories and connectivity-preserving functors. If the connectivity structure is compatible with cobase change and $\cC$ is differentiable, a dual statement holds: the coCartesian fibration $T_{\cC}\to \cC$ classifies a covariant functor from $\cC$ to stable categories and connectivity-preserving functors.
\end{rmk}

\begin{rmk}
  If $\cC$ is an $\infty$-topos it can be shown that $T_{\cC}$ is also an $\infty$-topos. In this case $T_{\cC}$ has the connectivity structure described in Example \ref{exam:topoi}. Note that this connectivity structure is very different from the connectivity structure on $T_{\cC}$ described in Proposition \ref{prop:tangentcatconnectivity}. More specifically, a morphism $f\co X\to Y$ in $T_{\cC}$ is $k$-connected in the toposic connectivity structure on $T_{\cC}$ exactly when $UX\to UY$ is $k$-connected in the toposic connectivity structure on $\cC$. 
\end{rmk}

\begin{prop}
  \label{prop:tangentskeletal}
  An object $A \in T_{\cC}$ is $k$-skeletal if and only if it is $k$-skeletal in the subcategory $\Sp(\cC_{/UA})$.
\end{prop}

\begin{proof}
  One direction is clear: objects which are skeletal in $T_\cC$ are still skeletal in the full subcategory.

  For the converse, suppose that $A$ is $k$-skeletal in $\Sp(\cC_{/UA})$, and that we have a $k$-connected map $g\co X \to Y$ in $T_\cC$. Let $f$ be a map $f\co A \to Y$ which we would like to lift, with associated base-change functor $f^*\co \Sp(\cC_{/UY}) \to \Sp(\cC_{/UA})$.

  The $k$-connected map $g$ determines an equivalence $UY \simeq UX$ and a $k$-connected map $X \to g^*Y \in \Sp(\cC_{/UX})$. The map $UA \to UY$ automatically lifts to a map $h\co A \to UX$, and a lift of $f$ is equivalent to asking for a lift of the map $A \to f^* Y = h^* g^* Y$ along $h^* X \to h^* g^* Y$ in $\Sp(\cC_{/UA})$. However, $h^* X \to h^* g^* Y$ is still $k$-connected, and so this lift exists.
\end{proof}

The forgetful functor $U\co T_\cC \to \cC$ has a left adjoint: the absolute cotangent functor $\mb L\co \cC \to T_\cC$, sending $A \in \cC$ to $\mb L(A) \in \Sp(\cC_{/A}) \subset T_\cC$ \cite[7.3.2.14]{lurie-higheralgebra}.

\begin{defn}\label{def:tangentiallynilpotent}
  A map $f\co X \to Y$ in $\cC$ is \emph{nilpotent} if it is $\Omega^\infty (E)$-nilpotent, where $E$ is the class of morphisms $f$ in $T_\cC$ whose underlying morphism $Uf$ is an equivalence.
\end{defn}

\begin{rmk}
We may think of such a morphism $X\to Y$ as being a map of abelian group objects in some slice category $\cC_{/Z}$. 
\end{rmk}

\noindent The following lemma gives a good supply of nilpotent morphisms in a presentable $\infty$-category:

\begin{lem}\label{lem:squarezeroisnilpotent}
If $\cC$ is presentable and a map $f\co A\to B$ is a square-zero morphism in the sense of \cite[7.4.1.6]{lurie-higheralgebra} then it is nilpotent. 
\end{lem}

\begin{proof}
Note that by \cite[7.4.1.7]{lurie-higheralgebra}, any square-zero extension $f\co A \to B$ by an object $M \in \Sp(\cC_{/B})$ is a pullback of a map $B\to B\oplus \Sigma M$. This latter map is $\Omega^\infty(0_B\to \Sigma M)$ where $0_B\to \Sigma M$ is the image of the zero morphism in $\Sp(\cC_{/B})$ under its inclusion (as the fiber over $B$) into $T_{\cC}$. Therefore $A\to B$ is the pullback of a nilpotent morphism.
\end{proof}

\begin{prop}\label{prop:skeletaliffcotangentskeletal}
  Suppose that all $k$-connected maps in $\cC$ are nilpotent. Then an object $A$ of $\cC$ is $k$-skeletal if and only if the absolute cotangent complex $\mb L(A)$ is $k$-skeletal in $\Sp(\cC_{/A})$.
\end{prop}

\begin{proof}
  The object $A$ is $k$-skeletal in $\cC$ if and only if $\mb L(A)$ is $k$-skeletal in $T_{\cC}$ by Proposition~\ref{cor:nilpotentadjunctions}. We can then deduce the result from Proposition~\ref{prop:tangentskeletal}.
\end{proof}

\begin{rmk}
Recall that $T_{\cC}\to \Fun(\Delta^1,\cC)$ is the \textit{stable envelope}, in the sense of \cite[7.3.1.1]{lurie-higheralgebra}, of the presentable fibration $\Fun(\Delta^1,\cC)\to \cC$ given by evaluating a morphism at its target. One can generalize Definition \ref{def:tangentiallynilpotent} and say that for any presentable fibration $p\co\cC\to \dD$ with stable envelope $q\co \cC'\to \cC$, an object of $\cC$ is nilpotent if it is in the essential image of $q$.  In the case that we take $p$ to be the presentable fibration $\Spaces\to\Delta^0$, the classically nilpotent spaces are nilpotent objects of $\Spaces$ (cf.~Example \ref{exam:classicalnilpotence}).
\end{rmk}

\section{Lifting constructions of skeleta}
\label{sec:lifting}

\subsection{Subductivity}

\begin{defn}
  Suppose that we have an adjunction
  \[
    \cC \mathop{\rightleftarrows}^L_R \dD.
  \]
  We say that the adjunction is $k$-\emph{subductive} on $\cC$ if, whenever a map $X \to Y$ in $\cC$ is $k$-connected, the square
  \[
    \xymatrix{
      X \ar[r] \ar[d] & Y \ar[d] \\
      RL(X) \ar[r] & RL(Y)
    }
  \]
  is always $k$-Cartesian, and \emph{strongly $k$-subductive} if the square is always $(k+1)$-Cartesian.

  The \emph{subduction zone} is the set of $k$ such that the adjunction is $k$-subductive. If it is subductive for all $k$, we simply refer to it as \emph{subductive}. (Similar definitions apply for terms such as the \emph{strong subduction zone}.)
\end{defn}

\begin{exam}\label{exam:sigmaomegasubductive}
  Consider the $\Sigma-\Omega$ adjunction on pointed spaces. The condition of $k$-subductivity asks whether, for a $k$-connected map $X \to Y$, the diagram
  \[
    \xymatrix{
      X \ar[r] \ar[d] & Y \ar[d]\\
      \Omega \Sigma X \ar[r] &
      \Omega \Sigma Y
    }
  \]
  is $k$-Cartesian. If the spaces are allowed to be disconnected, this is generically false. If we restrict attention to pointed connected spaces, this adjunction is subductive: subductivity is equivalent to asking for surjectivity of the map
  \[
    \pi_{k+1}(Y,X) \to \pi_{k+2}(\Sigma Y, \Sigma X),
  \]
  which is a consequence of the Blakers--Massey excision theorem for any $k \geq 1$. If we further restrict to pointed 1-connected paces, this adjunction is strongly subductive.
\end{exam}

\begin{exam}\label{exam:0connectedringmapgivessubductive}
  Suppose that $f\co R \to S$ is a map of connective ring spectra, determining an adjunction between $\LMod_R$ and $\LMod_S$. Then this adjunction is subductive if and only if, for any $k$-connected map $M \to N$, the square
  \[
    \xymatrix{
      M \ar[r] \ar[d] & N \ar[d] \\
      S \otimes_R M \ar[r] & S \otimes_R N
    }
  \]
  is $k$-Cartesian. Taking fibers, we find that subductivity is equivalent to showing that for any $(k-1)$-connected object $F$, the map $F \to S \otimes_R F$ is $k$-connected. On $\pi_k$, this is the map
  \[
    \pi_k(F) \to \pi_0(S) \otimes_{\pi_0(R)} \pi_k(F),
  \]
  Therefore, this adjunction is subductive if and only if the map $R \to S$ is $0$-connected. Similarly, the adjunction is strongly subductive if and only if the map $R \to S$ is $1$-connected.
\end{exam}

\begin{prop}\label{prop:strongsubductivereflects}
  Suppose that we have a strongly $k$-subductive adjunction
  \[
    \cC \mathop{\rightleftarrows}^L_R \dD
  \]
  such that $R$ preserves $(k+1)$-connectivity, and that $f\co X \to Y$ is $k$-connected in $\cC$. If $Lf\co LX \to LY$ is $(k+1)$-connected, then $f$ is $(k+1)$-connected.
\end{prop}

\begin{proof}
  The map $f$ factors as a composite
  \[
    X \to RL(X) \times_{RL(Y)} Y \to Y.
  \]
  The first map is $(k+1)$-connected by strong $k$-subductivity, and the second is a pullback of the $(k+1)$-connected map $RL(X) \to RL(Y)$.
\end{proof}

Applying the previous proposition inductively gives the following result.

\begin{cor}\label{cor:strongsubductivereflects}
  Suppose that we have an adjunction
  \[
    \cC \mathop{\rightleftarrows}^L_R \dD
  \]
  such that $\{k,k+1,\dots,N\}$ are in the strong subduction zone, that $f\co X \to Y$ in $\cC$ is $k$-connected, and that $L$ and $R$ preserve connectivity. Then $f$ is $N$-connected if and only if $Lf$ is $N$-connected.
\end{cor}

\begin{prop}
  \label{prop:compositeexcisive}
Let $k$ and $j$ be integers with $j\geq 0$.  Suppose that we have a (strongly) $k$-subductive adjunction
  \[
    \cC \mathop{\rightleftarrows}^L_R \dD
  \]
  and a (strongly) $(k+j)$-subductive adjunction
  \[
    \dD \mathop{\rightleftarrows}^M_S \eE,
  \]
  that $L$ increases connectivity by at least $j$, and that $R$ decreases connectivity by no more than $j$. Then the composite adjunction is also (strongly) $k$-subductive.
\end{prop}

\begin{proof}
  Suppose that $f\co X \to Y$ is $k$-connected. Let $\epsilon$ be $0$ in the case of subductivity, and $1$ in the case of strong subductivity; we wish to show $X \to Y \times_{RSML(Y)} RSML(X)$ is $(k+\epsilon)$-connected.

  From \cite[Proposition 2.1.9]{riehlverity-elements}, we know that the unit of the composite adjunction is given the composite transformation $id_{\cC}\to RL\to RSML$. Therefore we have a composite of cospans:
\[\begin{tikzcd}
	X && X \\
	Y & X & {RL(X)} \\
	Y & {RL(Y)} & {RSML(X)} \\
	& {RSML(Y)}
	\arrow["{id_X}", from=1-1, to=2-2]
	\arrow["{id_X}"', from=1-3, to=2-2]
	\arrow[from=2-1, to=3-2]
	\arrow[from=2-3, to=3-2]
	\arrow[from=3-3, to=4-2]
	\arrow[from=3-1, to=4-2]
	\arrow[Rightarrow, from=2-1, to=3-1]
	\arrow[Rightarrow, from=1-1, to=2-1]
	\arrow[Rightarrow, from=2-2, to=3-2]
	\arrow[Rightarrow, from=3-2, to=4-2]
	\arrow[Rightarrow, from=1-3, to=2-3]
	\arrow[Rightarrow, from=2-3, to=3-3]
\end{tikzcd}\]
which induces a factorization of the universal map $X\to Y\times_{RSML(Y)} RSML(X)$. Thus we need to prove that the composite map
  \[
    X\to
    Y \times_{RL(Y)} RL(X) \to
    Y \times_{RL(Y)} \left(RL(Y) \fibprod_{RSML(X)} RSML(X)\right)
    \simeq Y \times_{RSML(Y)} RSML(X)
  \]
  is $(k+\epsilon)$-connected. The first map is $(k+\epsilon)$-connected because the first adjunction is (strongly) $k$-subductive, and so it suffices to prove that the second map is $(k+\epsilon)$-connected.

  Because $L$ increases connectivity by at least $j$, the map $Lf\co LX \to LY$ is $(k+j)$-connected, and therefore the map
  \[
    LX \to LY \times_{SML(X)} SML(Y)
  \]
  is $(k+j+\epsilon)$-connected because the second adjunction is (strongly) $(k+j)$-subductive. Applying $R$, the map
  \[
    RL(X) \to RL(Y) \times_{RSML(X)} RSML(X)
  \]
  is $(k+\epsilon)$-connected. Taking the fiber product with $Y$ over $RL(Y)$, we find that the map
  \[
    Y \times_{RL(Y)} RL(X) \to Y \times_{RSML(X)} RSML(X)
  \]
  is $(k+\epsilon)$-connected, as desired.
\end{proof}

If the $\Sigma-\Omega$ adjunction is (strongly) subductive, we get the following result.

\begin{cor}
  \label{cor:subductiveiterate}
  Suppose $\cC$ is pointed and has pushouts, that connectivity is compatible with cobase change, and that $\{k,k+1,\dots,k+n-1\}$ are in the (strong) subduction zone of the $\Sigma-\Omega$ adjunction. Then the $\Sigma^n-\Omega^n$ adjunction is (strongly) $k$-subductive. 
\end{cor}

\begin{proof}
  The functors $\Omega$ and $\Sigma$ both shift connectivity by $1$ in opposing directions, the latter by Corollary~\ref{cor:freudenthal}. Therefore, iterated application of Proposition~\ref{prop:compositeexcisive} implies that the composite $\Sigma^n-\Omega^n$ adjunction is subductive.
\end{proof}

\begin{prop}
  \label{prop:subductivestabilization}
  Suppose that $\cC$ is differentiable, pointed, and admits pushouts. Suppose that the connectivity structure on $\cC$ is compatible with cobase change and sequential colimits.

  If the $\Sigma-\Omega$ adjunction is (strongly) subductive, then the $\Sigma^\infty-\Omega^\infty$ adjunction between $\cC$ and $\Sp(\cC)$ is (strongly) subductive.
\end{prop}

\begin{proof}
The proof of Proposition \ref{prop:stabilizationpreservesconnective} shows that $\Omega^\infty\Sigma^\infty X\simeq\colim_m\Omega^m\Sigma^m X$. Therefore we need to prove that, for any $k$-connected map $X \to Y$, the square
  \[
    \xymatrix{
      X \ar[r] \ar[d] & Y \ar[d] \\
      \colim_m \Omega^m \Sigma^m X \ar[r] &
      \colim_m \Omega^m \Sigma^m Y
    }
  \]
  is $(k+\epsilon)$-Cartesian. For any individual $m$, this diagram is $(k+\epsilon)$-Cartesian by Corollary~\ref{cor:subductiveiterate}. The fiber product, and the connectivity of the map from $X$ to it, are preserved by sequential colimits because $\cC$ is differentiable.
\end{proof}

\subsection{Excavation}

\begin{defn}
  Let $L\co \cC \to \dD$ be a colimit-preserving functor. Suppose that we have  a set $S$ of maps in $\cC$, a map $X \to Y$ in $\cC$, and a commutative diagram
  \[
    \xymatrix{
      \coprod_{t \in T} LA_{j_t} \ar[r] \ar[d] &
      LX \ar[d] \\
      \coprod_{t \in T} LB_{j_t} \ar[r] &
      LY
    }
  \]
  where the left-hand map is a coproduct of maps in $LS$. We say that this can be excavated from $\dD$ if the square diagram from each term in the coproduct is equivalent to the image of a diagram in $L$. In particular, the diagram lifts to a commutative diagram
  \[
    \xymatrix{
      \coprod_{t \in T} A_{j_t} \ar[r] \ar[d] &
      X\ar[d] \\
      \coprod_{t \in T} B_{j_t} \ar[r] &
      Y
    }
  \]
  in $\cC$, and the constructed factorization $LX \to Z \to LY$ through the pushout lifts to a constructed factorization $X \to W \to Y$.

  More generally, if we have sets of cells $S_k$ and a map $X \to Y$ with a factorization
  \[
    LX = Z^{(n-1)} \to Z^{(n)} \to Z^{(n+1)} \to \dots \to Z^{(N)} \to LY,
  \]
  where each map $Z^{(k-1)} \to Z^{(k)}$ is cellularly constructed from maps in $LS_k$, then we say that this can be excavated if each stage can be excavated: there exists a lift to a sequence
  \[
    X = W^{(n-1)} \to W^{(n)} \to W^{(n+1)} \to \dots \to W^{(N)} \to Y,
  \]
  in $\cC$, where each map $W^{(k-1)} \to W^{(k)}$ is cellularly constructed from the maps in $S_k$, lifting the construction of $Z^{(k)}$.
\end{defn}

\subsection{Excavating cells}

The following lemma assures us that, when we have a subductive adjunction, ``cells can be excavated so long as their ranges can.''
\begin{lem}
  Suppose that
  \[
    \cC \mathop{\rightleftarrows}^L_R \dD
  \]
  is a $k$-subductive adjunction and that $f\co X \to Y$ is a $k$-connected map in $\cC$. Then, for any $(k+1)$-cell $A \to B$ in $\cC$, a commutative diagram
  \[
    \xymatrix{
      LA \ar[r] \ar[d] & LX \ar[d] \\
      LB \ar[r] & LY
    }
  \]
  can be excavated from $\dD$ if and only if the map $LB \to LY$ lifts to $\cC$.
\end{lem}

\begin{proof}
  Showing that this diagram is the image of one in $L$ is the same as asking that we can complete the commutative diagram
  \[
    \xymatrix@R=1pc{
      A \ar[dd] \ar[rr] \ar@{.>}[dr] && RL(X) \ar[dd] \\
      & X \ar[ur] \ar[dd]\\
      B \ar[rr]|-\hole \ar@{.>}[dr]&& RL(Y)\\
      & Y \ar[ur].
    }
  \]
  By adjunction, lifting $LB \to LY$ to a map $B \to Y$ is the same as constructing the bottom commutative triangle.
  
  Suppose that a lift $B \to Y$ is chosen. Then we have a map $A \to B \times_{RL(Y)} RL(X) \to Y \times_{RL(Y)} RL(X)$, and what remains is to lift it along the map $X \to Y \times_{RL(Y)} RL(X)$. However, this map is $k$-connected by assumption and $A$ is $k$-skeletal, so such a lift exists.
\end{proof}

\begin{rmk}
  The lifting requirement holds automatically in several important cases, such as if the object $B$ is a zero object of $\cC$. Our principal case will use boundedness: lifting holds if the cell is $j$-bounded and the unit $Y \to RL(Y)$ is $j$-connected.
\end{rmk}

\begin{cor}
  \label{cor:excavateattachment}
  Fix a $k$-subductive adjunction $\cC \mathop{\rightleftarrows}^L_R \dD$. Let $S_{k+1}$ be a set of $j$-bounded $(k+1)$-cells in $\cC$.

  Suppose that $Y \to RL(Y)$ is $j$-connected, that we have a $k$-connected map $f\co X \to Y$ in $\cC$ and a choice of factorization $LX \to Z \to LY$ where $LX \to Z$ is cellularly constructed from the maps in $LS_{k+1}$. Then $LX \to Z \to LY$ can be excavated from $\dD$.
\end{cor}

\subsection{Excavating skeleta}

Our goal in this section is to expand the previous section inductively, and show that by using subductive adjunctions $\cC \rightleftarrows \dD$, we can lift cellular constructions from $\dD$. In particular, this will allow us to construct minimal skeleta.

\begin{thm}
  \label{thm:excavateskeleta}
  Suppose that we have an adjunction
  \[
    \cC \mathop{\rightleftarrows}^L_R \dD
  \]
   such that $\{n, n+1, \dots, N-1\}$ are in the subduction zone.  Fix sets $S_k$ of $j$-bounded $k$-cells in $\cC$ for $n < k \leq N$.

  Suppose $Y \to RL(Y)$ is $j$-connected, that we have an $n$-connected map $f\co X \to Y$ in $\cC$, and a cellular construction $LX = Z^{(n)} \to Z^{(n+1)} \to \dots \to Z^{(N)} \to LY$ using the cells in $LS_k$. Then the cellular construction of $Z^{(N)}$ can be excavated from $\dD$ to a sequence $X = W^{(n)} \to W^{(n+1)} \to \dots \to W^{(N)} \to Y$.

  If $L$ reflects $N$-connectivity, then the map $W^{(N)} \to Y$ is an $N$-skeleton if the original map $Z^{(N)} \to Y$ was. If $L$ also reflects $(N+1)$-connectivity, this skeleton is minimal if the original map was.
\end{thm}

\begin{proof}
  Starting with $W^{(n)} = X$, inductively apply Corollary~\ref{cor:excavateattachment} to excavate the cellular constructions of $(LW)^{(k-1)} \simeq Z^{(k-1)} \to Z^{(k)} \to LY$ to cellular constructions of $W^{(k-1)} \to W^{(k)} \to Y$.

  Because the $S_k$ are $k$-cells, each map $W^{(k-1)} \to W^{(k)}$ is relatively $k$-cellular.

  If $L$ reflects connectivity appropriately, and the map $(LW)^{(N)} \simeq Z^{(N)} \to LY$ is $N$-connected, then the map $W^{(N)} \to Y$ is $N$-connected and hence an $N$-skeleton. If the original was a minimal $N$-skeleton, then this connectivity estimate improves, making $W^{(N)}$ into a minimal $N$-skeleton of $Y$.
\end{proof}

\begin{exam}
  The $\Sigma-\Omega$ adjunction between $0$-connected pointed spaces and $1$-connected pointed spaces has subduction zone $[1,\infty)$ as in Example~\ref{exam:sigmaomegasubductive}. The connectivity structure in both cases is determined by $0$-bounded cells $S^k \to D^{k+1}$, which are preserved by suspension.

  For a path-connected space $Y$, the map $Y \to \Omega \Sigma Y$ is $0$-connected. Given a 1-connected map $X \to Y$ between 0-connected spaces, we can therefore excavate any relative CW-factorization $\Sigma X \to Z^{(2)} \to \dots \to Z^{(N+1)} \to \Sigma Y$ to a relative CW-factorization $X \to W^{(1)} \to \dots \to W^{(N)} \to Y$.

  On $1$-connected spaces, $\Sigma$ reflects connectivity. Therefore, if $Y$ is $1$-connected, any cellular construction of an $(N+1)$-skeleton for $\Sigma Y$ can be excavated to a cellular construction of an $N$-skeleton for $Y$.
\end{exam}

\section{Excision}
\label{sec:excision}

In this section we will show that many of these important properties hold for categories of algebras and modules over an operad in spectra. The fundamental tools to prove this are excision theorems, and excision theorems are hard work. For our applications, we will be appealing to the work of Ching--Harper on excision for algebras in a category of module spectra \cite{ching-harper-excision}. However, our use of their main result is simple enough that we can axiomatize it in this section.

\subsection{Excisive connectivity structures}
\label{sec:algebraicexcision}

\begin{defn}
  \label{def:excisive}
  Fix $\cC$ with homotopy pushouts and an initial object, with a connectivity structure that is compatible with cobase change. We say that the connectivity structure on $\cC$ \emph{satisfies algebraic excision} if the following conditions hold.

  \begin{enumerate}
  \item All maps $X \to Y$ in $\cC$ are $(-1)$-connected.
  \item Let $W$ be a nonempty finite set with power set $\mathcal{P}(W)$, viewed as a poset, and $X\co \mathcal{P}(W) \to \cC$ a ``$W$-cube'' in $\cC$. Suppose that:
    \begin{enumerate}
    \item for each nonempty subset $V \subset W$, the $V$-cube $X|_{\mathcal{P}(V)}$ is $k_V$-coCartesian, and
    \item $k_U \leq k_V$ for all $U \subset V \subset W$.
    \end{enumerate}
    Then $X$ is $k$-Cartesian, where $k$ is the minimum of $-|W| + \sum_{V \in \lambda} (k_V + 1)$ over all partitions $\lambda$ of $W$ by nonempty finite sets.
  \end{enumerate}
\end{defn}

\begin{rmk}
  The category $\Spaces$ of spaces does not satisfy this type of excision: excision estimates for spaces differ by a shift.\footnote{In principle, pointed connected spaces are equivalent to topological groups via the loop space/classifying space relationship, and topological groups \emph{do} satisfy algebraic excision; this can help make sense of the shift in connectivity.}
\end{rmk}

\begin{exam}\label{exam:excisionnnuke}
  Suppose $\mf O$ is a connective operad in the category of modules over a connective commutative symmetric ring spectrum $S$, with category $\Alg_{\mf O}$ of \emph{connective} algebras. Then \cite[Theorem 1.7]{ching-harper-excision} is precisely that $\Alg_{\mf O}$ satisfies algebraic excision. This includes models for the category of $\mb E_n$ $S$-algebras, where $S$ is a commutative ring spectrum.\footnote{The paper \cite{ching-harper-excision} is written in terms of operads $\mf O$ in symmetric spectra, using the Quillen adjunction between positive stable model structures on $\Mod_S$ and $\Alg_{\mf O}$. However, passage to the associated $\infty$-category preserves structures of rings, modules, and algebras, and the definitions of Cartesian and coCartesian cubes are defined using homotopy limits and colimits, which are equally well computed in $\Alg_{\mf O}$ or the associated $\infty$-category. Moreover, in the case where $\mf O$ is the suspension spectrum of an ordinary operad $\mathcal{O}$, algebras can be rectified: the $\infty$-category associated to $\Alg_{\mf O}$ is a model for the category of algebras for the associated $\infty$-operad.}
\end{exam}

\begin{rmk}
  The property of having algebraic excision is inherited by slice and coslice categories, because the colimits and limits of deleted cubical diagrams are calculated in the underlying $\infty$-category.
\end{rmk}

\subsection{Excision and augmentations}

\begin{prop}
  Suppose that $\cC$ satisfies algebraic excision and has a final object, and that the map $\varnothing \to \ast$ from the initial object to the final object is $j$-connected. Then there is an adjunction
  \[
    \cC \rightleftarrows \cC_{\ast}
  \]
  that is subductive if $j \geq 0$, and strongly subductive if $j > 0$.
\end{prop}

\begin{proof}
  The left adjoint is the functor $(-)_+ = (-) \amalg \ast$. By pushing out along the map $\varnothing \to \ast$, we find that $X \to X_+$ is always $j$-connected. As a result, consider the coCartesian square
  \[
    \xymatrix{
      X \ar[r] \ar[d] & Y \ar[d] \\
      X_+ \ar[r] & Y_+,
    }
  \]
  viewed as a $2$-cube. If the map $X \to Y$ is $k$-connected, then algebraic excision shows that the square is $n$-Cartesian, where $n = \min\{-2+(k+1)+(j+1), -2+\infty\} = k+j$, as desired.
\end{proof}

\begin{cor}\label{cor:unitconnected}
  Suppose that $\cC$ satisfies algebraic excision and that $\varnothing \to Z$ is $j$-connected. Then the adjunction $\cC_{/Z} \rightleftarrows (\cC_{/Z})_\ast$, between objects over $Z$ and objects augmented over $Z$, is subductive if $j \geq 0$ and strongly subductive if $j > 0$. The unit $X \to X \amalg Z$ is always $j$-connected.
\end{cor}

\subsection{Excision and suspension}

\begin{prop}\label{prop:excisionsuspension}
  Suppose that $\cC$ satisfies algebraic excision and that $X \to Z$ is a $k$-connected map in $\cC$, where $k \geq 0$. Then the map $X \to \Omega_Z \Sigma_Z X$ is $2k$-connected.
\end{prop}

\begin{proof}
  Consider a homotopy pushout diagram
  \[
    \xymatrix{
      X \ar[r] \ar[d] & Z \ar[d] \\
      Z \ar[r] & \Sigma_Z X,
    }
  \]
  viewed as a $2$-cube. The two arrows $X \to Z$ are $k$-connected by assumption, and the full cube is a pushout cube and hence is $\infty$-coCartesian. Algebraic excision then shows that it is $n$-Cartesian where $n = \min\{-2+(k+1)+(k+1), -2+\infty\} = 2k$. By definition, this means that the map $X \to \Omega_Z \Sigma_Z X$ is $2k$-connected.
\end{proof}

\begin{rmk}
  In particular, this implies that the map $\Omega_Z \Sigma_Z X \to Z$ is still $k$-connected.
\end{rmk}

\begin{prop}
  \label{prop:suspensionsubductive}
  Suppose that $\cC$ satisfies algebraic excision. Let $X \to Y \to Z$ be maps, where $X \to Y$ is $k$-connected and the maps $X \to Z$ and $Y \to Z$ are $\ell$-connected, for $k+1 \geq \ell \geq 0$. Then the diagram
  \[
    \xymatrix{
      X \ar[r] \ar[d] & Y \ar[d] \\
      \Omega_Z \Sigma_Z X \ar[r] &
      \Omega_Z \Sigma_Z Y
    }
  \]
  is $(k+\ell)$-Cartesian.
\end{prop}

\begin{proof}
  Consider the cubical diagram:
  \[
    \xymatrix@R=1pc{
      X \ar[rr] \ar[dd] \ar[dr] &&
      Z \ar[dr] \ar[dd]|\hole \\
      & Z \ar[rr] \ar[dd] &&
      \Sigma_Z X \ar[dd] \\
      Y \ar[rr]|\hole \ar[dr] &&
      Z \ar[dr] \\
      & Z \ar[rr] &&
      \Sigma_Z Y
    }
  \]
  The natural map from $X$ to the total homotopy pullback of the cube with the initial vertex deleted is equivalent to the map from $X$ to the homotopy pullback in the proposition statement, and hence it suffices for us to show that this cube is $(k+\ell)$-Cartesian.

  The maps $X \to Z$ are $\ell$-connected, and the map $X \to Y$ is $k$-connected. The top and bottom faces are homotopy pushout diagrams, and hence $\infty$-coCartesian; in particular, the whole cube is $\infty$-coCartesian by \cite[3.8(b)]{ching-harper-excision}. Because the map $X \to Y$ is $k$-connected and the map $Z \to Z$ is $\infty$-connected, the left-hand and back faces are both $(k+1)$-coCartesian, again by \cite[3.8(b)]{ching-harper-excision}.

  The inclusion of each initial face into a larger face improves connectivity, and so the excision estimate applies. The whole cube is $n$-Cartesian, where
  \[
    n = \min\{-3+(k+1)+(\ell+1)+(\ell+1),
    -3+(k+1)+\infty,-3+(\ell+1)+(k+2),-3+\infty\} = k+\ell,
  \]
  as desired.
\end{proof}

\begin{cor}\label{cor:subductivealgebras}
If $\cC$ satisfies algebraic excision, the $\Sigma_Z-\Omega_Z$ adjunction is subductive when restricted to objects $X$ such that $X \to Z$ is $0$-connected, and strongly subductive when restricted to objects such that $X \to Z$ is $1$-connected.
\end{cor}

\begin{prop}\label{prop:reflectivealgebras}
  Suppose that $\cC$ satisfies algebraic excision, and that $X \to Y$ is a map of objects over $Z$ such that $X \to Z$ and $Y \to Z$ are $1$-connected. Then, for $k \geq 0$, the map $X \to Y$ is $k$-connected if and only if the map $\Omega_Z \Sigma_Z X \to \Omega_Z \Sigma_Z Y$ is $k$-connected.
\end{prop}

\begin{proof}
  Because the maps $X \to Z$ and $Y \to Z$ are $1$-connected, the map $X \to Y$ is $0$-connected. Therefore, we can apply Proposition~\ref{prop:strongsubductivereflects}.
\end{proof}

\section{Applications}
\label{sec:applications}

\subsection{Skeleta in derived categories}

In this section we will examine what skeletality means for an object in the classical derived category of a ring $R$, with the connectivity structure determined by the $t$-structure as in Example \ref{exam:spectratstructure}.

\begin{defn}
  \label{def:amplitude}
  Let $R$ be an ordinary ring. We say that a complex of $R$-modules has \emph{projective amplitude in $[a,b]$} if it is equivalent to a complex of projectives concentrated in degrees $a$ through $b$.
\end{defn}

\begin{prop}
  \label{prop:skeletalamplitude}
  A complex of $R$-modules has projective amplitude in $[a,b]$ if and only if it is $(a-1)$-connected and $b$-skeletal.
\end{prop}

\begin{proof}
  Suppose $A$ is a complex with the given projective amplitude; since the properties we need to show are invariant under equivalence, we may assume that $A$ is a complex of projectives concentrated in degrees $a$ through $b$. Then its homology is clearly concentrated in degrees $a$ through $b$, so it is $(a-1)$-connected. Moreover, if $X \to Y$ is a $b$-connected map of complexes, its cofiber $Y/X$ is $(b+1)$-connected, and hence $[A,Y/X] = 0$ via a hypercohomology spectral sequence with $E_1$-term
  \[
    \Hom_R(A_s, H_t(Y/X)) \Rightarrow H_{t-s} \Hom(A,Y/X).
  \]
  This implies the desired lifting property by the long exact sequence for $[A,-]$.
  
  To prove the converse, we proceed by induction on $b-a$. By first applying a shift $\Sigma^{-a}$, we may assume without loss of generality that $a=0$.

  For the base case, suppose that $A$ is $(-1)$-connected and $0$-skeletal; we wish to show that $A$ is equivalent to a projective complex concentrated in degree $0$. Without loss of generality we may assume $A$ is a complex of projectives in nonnegative degrees. Any surjective map of (discrete) $R$-modules $M \to N$ can be viewed as a $0$-connected map of complexes concentrated in degree $0$. The map $[A,M] \to [A,N]$ is therefore surjective; however, this is isomorphic to the map
  \[
    \Hom_R(H_0 A, M) \to \Hom_R(H_0 A, N).
  \]
  Since $M \to N$ was an arbitrary surjection, $H_0(A)$ is a projective module $P$. Both $A$ and $P$ are $0$-skeletal, and the maps $A \to P$ and $P \to P$ are both $1$-connected. This makes both $A$ and $P$ minimal $0$-skeleta of $P$, and hence equivalent by Proposition~\ref{prop:skeletaluniqueness}.

  Now suppose by induction that we have shown the result for $b-1$. Given a complex $A$ which is $(-1)$-connected and $b$-skeletal, let $P \to H_0(A)$ be a surjection from a projective module, with a lift to a map $P \to A$. Then the cofiber $A/P$ is $0$-connected, and it is $b$-skeletal by Proposition~\ref{prop:pushoutskeletal}. By induction it then has projective amplitude in $[1,b]$. The complex $A$ is equivalent the cofiber of the map $\Sigma^{-1} A/P \to P$, and thus has projective amplitude in $[0,b]$.
\end{proof}

\begin{prop}
  \label{prop:PIDskeletal}
  Suppose that $R$ is a principal ideal domain. Then a complex $M$ of $R$-modules is $k$-skeletal if and only if $H_k(M)$ is free and $H_* M = 0$ for $* > k$.
\end{prop}

\begin{proof}
  Since principal ideal domains have projective dimension $1$, any complex of $R$-modules splits: there is an equivalence
  \[
    \bigoplus_d \Sigma^d H_d(M) \simeq M.
  \]
  Therefore, $M$ is $k$-skeletal if and only if each $\Sigma^d H_d(M)$ is $k$-skeletal by Proposition~\ref{prop:coproductskeletal}, which is true if and only if $H_d(M)$ is $(k-d)$-skeletal.

  By the previous proposition, this is equivalent to $H_d(M)$ having projective amplitude in $[0,k-d]$. Since $R$ has projective dimension $1$, every discrete module has projective amplitude in $[0,1]$, and so this is automatically satisfied for $d \leq k-1$. When $d=k$, we must have that $H_k(M)$ has projective amplitude in $[0,0]$, which is equivalent to asking that $H_k(M)$ is projective, and hence free. When $d > k$, we must have that $H_k(M)$ is trivial.
\end{proof}

\begin{cor}
  \label{cor:fieldskeletal}
  Suppose that $R$ is a field. Then a complex $M$ of $R$-modules is
  $k$-skeletal if and only if $H_*(M) = 0$ for $* > k$.
\end{cor}

We now consider the construction of minimal skeleta.
\begin{prop}
  Suppose that $A \to M$ is a map of chain complexes over $R$ that is a $(k-1)$-skeleton, with cofiber $M/A$. Then there exists a minimal $k$-skeleton $B \to M$ if and only if $H_k(M/A)$ is a projective $R$-module and $H_{k+1}(M/A) = 0$.
\end{prop}

\begin{proof}
  Suppose that $B \to M$ is a minimal $k$-skeleton. Then there exists a unique lift $A \to B$ over $M$ by Corollary~\ref{cor:uniquelifts}. The object $B/A$ is $(k-1)$-connected and $k$-skeletal by Proposition~\ref{prop:pushoutskeletal}, and hence equivalent to a shift $\Sigma^k P$ of a projective $R$-module.

  Because of the identification
  \[
    (M/A) / (B/A) \simeq M/B,
  \]
  the map $B \to M$ is $(k+1)$-connected if and only if the map $B/A \to M/A$ is $(k+1)$-connected. However, for the map $\Sigma^k P \to M/A$ to be $(k+1)$-connected we must have that $H_k(M/A) = P$ and $H_{k+1}(M/A) = 0$.

  Conversely, suppose that $H_k(M/A)$ is a projective module $P$ and that $H_{k+1}(M/A) = 0$. Then, since $M/A$ is $(k-1)$-connected, there is a $(k+2)$-connected map $M/A \to \Sigma^k P$, which has a section because $\Sigma^k P$ is $k$-skeletal; the section is $(k+1)$-connected. Let $C$ be the cofiber of the map $\Sigma^k P \to M/A$; it is $(k+1)$-connected. Let $B$ be the fiber of the map $M \to C$; the map $B \to M$ is $(k+1)$-connected. The octahedral axiom implies that there is a cofiber sequence
  \[
    \Sigma^{k-1} P \to A \to B
  \]
  and hence $B$ is $k$-skeletal. Therefore, $B$ is a minimal $k$-skeleton of $M$.
\end{proof}

\begin{cor}
  Suppose $R$ is a principal ideal domain. Then a chain complex of $R$-modules $M$ admits a minimal $k$-skeleton if and only if $H_k(M)$ is free and $H_{k+1}(M) = 0$.
\end{cor}

\begin{proof}
  Let $A$ be any $(k-1)$-skeleton of $M$. Then $H_{k+1}(M/A) = H_{k+1}(M)$, and so these groups are either both zero or both nonzero. There is also an exact sequence
  \[
    0 \to H_k(M) \to H_k(M/A) \to H_{k-1}(A).
  \]
  The group $H_{k-1}(A)$ is free by Proposition~\ref{prop:PIDskeletal}, and since $R$ is a principal ideal domain the image of the map $H_k(M/A) \to H_{k-1}(A)$ is also free. Therefore, the group $H_k(M/A)$ splits as a direct sum of $H_k(M)$ and a free module, and so $H_k(M)$ is projective if and only if $H_k(M/A)$ is.
\end{proof}

Because there is an equivalence between complexes of $R$-modules and modules over the Eilenberg--Mac Lane spectrum $HR$ \cite[5.1.6]{schwede-shipley-stablemodules}, taking homology groups to homotopy groups, we arrive at the following.

\begin{prop}
  \label{prop:EMskeletal}
  Suppose $R$ is a principal ideal domain. An $HR$-module $M$ is $k$-skeletal if and only if $\pi_k(M)$ is projective and $\pi_*(M) = 0$ for $* > k$. An $HR$-module $M$ admits a minimal $k$-skeleton if and only of $\pi_k(M)$ is free and $\pi_{k+1}(M) = 0$.

  In particular, suppose $R$ is a field. An $HR$-module $M$ is $k$-skeletal if and only if $\pi_*(M) = 0$ for $\ast > k$. An $HR$-module $M$ admits a minimal $k$-skeleton if and only if $\pi_{k+1}(M) = 0$.
\end{prop}

\begin{cor}\label{cor:PIDskeleton}
  If $R$ is a principal ideal domain, then any $N$-skeleton of an $(n-1)$-connected $HR$-module $M$ admits a cellular construction using the cells $(\{\Sigma^{k-1} HR \to \ast\})_{n \leq k \leq N}$.
\end{cor}

\subsection{Skeleta for spaces}

For $k \geq 2$, any $k$-connected map of spaces $X \to Y$ is nilpotent: the relative Postnikov tower expresses each path component as a limit of pullbacks
\[
  \xymatrix{
    P_n (Y,X) \ar[r] \ar[d] & K(\pi_1 Y, 1) \ar[d] \\
    P_{n-1} (Y,X) \ar[r] & K(\pi_1 Y, 1) \ltimes K(\pi_n (Y,X), n+1).
  }
\]
As a result, $X$ is $k$-skeletal for $k \geq 2$ if and only if its image in the stable category $\Sp(\Spaces_{/X})$ is $k$-skeletal.

Following Waldhausen, this stable category is equivalent to the category of functors $X \to \Sp$, and the image of $X$ is the constant functor with value $\mb S$. In the case where $X$ is connected, this is equivalent to the category of modules over the spherical group algebra $\mb S[\Omega X]$, and the image of $X$ is the trivial module $\mb S$.

The Postnikov truncation map $\mb S[\Omega X] \to H\mb Z[\pi_1X]$ is $1$-connected. As in Example~\ref{exam:pi0skeletal} we find that $X$ is $k$-skeletal if and only if the left $H\mb Z[\pi_1X]$-module
\[
  \mb S \otimes_{\mb S[\Omega X]} H\mb Z[\pi_1 X] \simeq H\mb Z
  \otimes \widetilde X
\]
has projective amplitude in $[0,k]$, where $\widetilde X$ is the universal cover. This is equivalent to asking that the complex $C_*(\widetilde X)$ of $\mb Z[\pi_1 X]$ has projective amplitude in $[0,k]$. This is related, but somewhat orthogonal, to Wall's finiteness obstruction.

\subsection{Algebras in module categories}

Throughout this section we assume that $S$ is a commutative ring spectrum and $\Alg_{\mb E_n}(S)$ is the $\infty$-category of $\mb E_n$-algebras in left $S$-modules. Recall that there is a connectivity structure on $\Alg_{\mb E_n}(S)_{\geq 0}$, lifted from $\LMod_S$, that is described in Examples \ref{exam:connectivityforcommutativealgebras} and \ref{exam:connectivityforalgebras}. In particular, this connectivity structure is determined by the ($0$-bounded) cells $\mb T^S_{\mb E_n}(\Sigma^{k-1} S)\to S$ for $k\geq 1$ and $S\to \mb T_{\mb E_n}^S(S)$, where $\mb T_{\mb E_n}^S\co \LMod_S \to \Alg_{\mb E_n}(\LMod_S)$ is the free $\mb E_n$ $S$-algebra functor. Moreover, Proposition \ref{prop:adjointspreservecells} and the fact that the connectivity structure on $\Alg_{\mb E_n}(S)$ is compatible with cobase change imply that these cells are sufficient for $k$-skeleta for all $k\geq 0$, in the sense of Section \ref{sec:cellularskeleta}.

As discussed in Example~\ref{exam:excisionnnuke}, the category $\Alg_{\mb E_n}(S)$ satisfies algebraic excision by \cite[Theorem 1.7]{ching-harper-excision}, and hence all the results of Section~\ref{sec:excision} apply.

\begin{prop}\label{prop:algebrapostnikov}
  Any connective $\mb E_n$ $S$-algebra $A$ has a (convergent) Postnikov tower
  \[
    \dots \to P_2 A \to P_1 A \to P_0 A
  \]
  in $\mb E_n$-algebras, where each stage is given by a pullback diagram
  \[
    \xymatrix{
      P_{m+1} A \ar[r] \ar[d] & P_m A \ar[d] \\
      P_m A \ar[r] & P_m A \oplus \Sigma^{m+2} H\pi_{m+1} A
    }
  \]
  in which the right-hand map is $\Omega^\infty$ of a map in $\Sp(\Alg_{\mb E_n}(R)_{/P_m A})$. In particular, the maps $P_{m+1} A \to P_m A$ are nilpotent for $m > 0$.
\end{prop}

\begin{proof}
It follows from \cite[7.1.3.19]{lurie-higheralgebra} that every connective  $\mb E_n$ $S$-algebra has a Postnikov tower. Moreover, by Lemma \ref{lem:squarezeroisnilpotent} and \cite[7.4.1.28]{lurie-higheralgebra} we have that each of the morphisms $P_{m+1}A\to P_{m}A$ is nilpotent. It remains to show that each level of this tower can be obtained by the given pullback. By \cite[7.4.1.7]{lurie-higheralgebra} we have that $P_{n+1}A\to P_{m}A$ is \textit{some} pullback of a nilpotent morphism. To check that it is a pullback specifically of a morphism $P_{m}A\to P_mA\oplus\Sigma^{m+2}H\pi_{m+1}A$ it suffices to identify the fiber of $P_{m+1}A\to P_mA$, which is $\Sigma^{m+1}H\pi_{m+1}A$ by \cite[7.1.3.14]{lurie-higheralgebra} parts (3) and (5). 
\end{proof}

\begin{prop}
Suppose $f\co A\to B$ is a morphism of $\mb E_n$ $S$-algebras spectra such that $\pi_0(f)$ is surjective and has nilpotent kernel. Then $f$ is nilpotent. 
\end{prop}

\begin{proof}
Recall that a surjective map of discrete rings $R\to R'$ with nilpotent kernel $J$ can always be written as a limit of square-zero extensions: the tower of quotients $\dots \to R/J^3 \to R/J^2 \to R/J \cong R'$. By \cite[7.4.1.21]{lurie-higheralgebra}, if $P_0A\to P_0B$ is a square-zero extension of discrete rings then it is square-zero as a map of $\mb E_n$ $S$-algebras and so, by Lemma \ref{lem:squarezeroisnilpotent}, is nilpotent. Therefore the map $\pi_0(f)\co P_0A\to P_0B$ is nilpotent. 

Now, using the Postnikov towers of $A$ and $B$, factor the map $A \to B$ as a transfinite composite
  \[
    A \to \dots \to P_2 A \times_{P_2 B} B \to P_1 A \times_{P_1 B} B \to P_0 A \times_{P_0 B} B \to B.
  \]
  Notice that $\holim_m(P_mA\times_{P_mB} B)\simeq A\times_B B$, so $A$ is the homotopy limit of the above tower. Therefore, it suffices to show that all of the maps in the composition are nilpotent.
  
  The last map is the base-change of the map $P_0 A \to P_0 B$ along the map $B \to P_0 B$, and so it is nilpotent by the discrete case proven above.

  Each of the remaining maps factors as
  \[
    P_{m+1} A \times_{P_{m+1} B} B \to P_{m+1} A \times_{P_m B} B \to P_m A \times_{P_m B} B.
  \]
  The second map is the base-change of the composite $P_{m+1} A \to P_m A\to P_mB$ along the map $B \to P_m B$, and so it is nilpotent. The first map is the base-change of the map
  \[
    P_{m+1} B \to P_{m+1} B \times_{P_m B} P_{m+1} B
  \]
  along $P_{m+1} A \times_{P_m B} B \to P_{m+1} B \times_{P_m B} P_{m+1} B$, and so it suffices to show that $P_{m+1} B \to P_{m+1} A \times_{P_m B} P_{m+1} B$ is nilpotent.

  However, $P_{m+1} B$ is a square-zero extension of $P_m B$ via a map that becomes trivial when restricted to $P_{m+1} B$. Therefore, the pullback is the trivial square-zero extension
  \[
    P_{m+1} B \oplus \Sigma^{m+1} H\pi_{m+1} B \to P_{m+1} B,
  \]
  and the map $P_{m+1} B \to P_{m+1} B \oplus \Sigma^{m+1} H\pi_{m+1} B$ is the image of the map $0 \to \Sigma^{m+1} H\pi_m B$ in the stable category $\Sp((\Alg_{\mb E_n})_{/P_{m+1} B})$.
\end{proof}

\begin{cor}\label{cor:connectiveisnilpotent}
A $k$-connected map of~ $\mb E_n$ $R$-algebras $f\co A\to B$ is nilpotent whenever $k>0$. 
\end{cor}

Recall that the stable category $\Sp(\Alg^{\mb E_n}(S)_{/A})$ is identified with the category of ($S$-linear) $\mb E_n$ $A$-modules \cite[7.3.4.18]{lurie-higheralgebra}.

\begin{defn}\label{def:loday}
  For a commutative ring spectrum $S$ and an $\mb E_n$ $S$-algebra $A$, we define $\loday{S}{A}$ to be the relative Loday construction, or relative factorization homology object:
  \[
    \loday{S}{A} = S \tens_{\int_{S^{n-1}} S} \int_{S^{n-1}} A.
  \]
\end{defn}
The category of $S$-linear $\mb E_n$ $A$-modules is equivalent to the category of modules over $\loday{S}{A}$ \cite[7.3.5.3]{lurie-higheralgebra}.

By combining Corollary \ref{cor:connectiveisnilpotent} with Proposition \ref{prop:skeletaliffcotangentskeletal} we immediately obtain the following result:

\begin{prop}
  A connective $\mb E_n$ $S$-algebra $A$ is $k$-skeletal for $k > 0$ if and only if its absolute cotangent complex $\mb L^{ \mb E_n}(A)$ is $k$-skeletal in the stable category $\LMod_{\loday{S}{A}}$.
\end{prop}

We recall the following formula for this absolute cotangent complex.

\begin{thm}[{\cite[2.26]{francis-tangentcomplex},
    \cite[7.3.5.1]{lurie-higheralgebra}}]
  \label{thm:cotangentcofiber}
  There is a cofiber sequence
  \[
    \loday{S}{A} \to A \to \Sigma^n \mb L^{\mb E_n}(A).
  \]
  of $\loday{S}{A}$-modules.
\end{thm}

\begin{cor}
  For $k \geq 1-n$, the absolute cotangent complex $\mb L^{\mb E_n}(A)$ is $k$-skeletal if $A$ is a $(k+n)$-skeletal $\loday{S}{A}$-module.
\end{cor}

\begin{proof}
  The object $\Sigma^{-n} \loday{S}{A}$ is always a $(-n)$-skeletal module and by assumption $\Sigma^{-n} A$ is a $k$-skeletal module, and so this follows by Proposition~\ref{prop:pushoutskeletal}.
\end{proof}

\begin{rmk}
  For $k \geq 2-n$, the converse is also true.
\end{rmk}

If $A$ is connective, the Loday construction $\loday{S}{A}$ is connective. Under these circumstances, Example~\ref{exam:pi0skeletal} now gives us the following result.

\begin{thm}
  \label{thm:Enskeletality}
  A connective $\mb E_n$ $S$-algebra $A$ is $k$-skeletal for $k \geq 1$ if the $H\pi_0\left(\loday{S}{A}\right)$-module
  \[
    H\pi_0\left(\loday{S}{A}\right) \tens_{\loday{S}{A}} A
  \]
  is $(n+k)$-skeletal.
\end{thm}

When $n=1$, $\loday{S}{A}= A \otimes_S A^\op$ as an algebra, and $\pi_0\left(\loday{S}{A}\right) = \pi_0(A) \otimes_{\pi_0(S)} \pi_0(A)^\op$. The above theorem specializes to a statement about relative topological Hochschild homology.

\begin{prop}
  \label{prop:THHskeletality}
  Suppose $A$ is a connective $\mb E_1$ $S$-algebra and that $\pi_0(A)$ is a localized quotient of $\pi_0 S$ (in particular, a commutative ring). Then $A$ is $k$-skeletal for $k \geq 1$ if the topological Hochschild homology object
  \[
    \THH^S(A; H\pi_0 A) \simeq H\pi_0(A) \otimes_{A \otimes_S
      A^\op} A
  \]
  is a $(k+1)$-skeletal $H\pi_0(A)$-module.
\end{prop}

This allows the detection of skeleta. We next turn our attention to the construction of skeleta and minimal skeleta.

\begin{lem}\label{lem:lodaypi0}
Let $S$ be a commutative ring spectrum and $A$ an $\mb E_n$ $S$-algebra for $n\geq 1$ with $\pi_0(A)$ commutative. If the unit map $S\to A$ is $0$-connected then there is an isomorphism $\pi_0(A)\simeq \pi_0(\loday{S}{A})$.
\end{lem}

\begin{proof}
When $n>1$, the result follows immediately from the long exact sequence in homotopy groups applied to the cofiber sequence of Theorem \ref{thm:cotangentcofiber} and does not require the $0$-connectivity assumption on the unit nor the commutativity condition on $\pi_0(A)$. When $n=1$ however we use both conditions along with the isomorphism $\pi_0\left(\loday{S}{A}\right) = \pi_0(A) \otimes_{\pi_0(S)} \pi_0(A)^\op$ to deduce the result.
\end{proof}

\begin{lem}
If $A$ is an $\mb E_n$ $S$-algebra for a commutative ring spectrum $S$ and the unit $S\to A$ is $0$-connected then the unit $A\to\Omega^\infty\Sigma_+^\infty A\simeq A\oplus \mb L_A$ is $0$-connected.
\end{lem}

\begin{proof}
First note that, because $S\to A$ is $0$-connected, Corollary \ref{cor:unitconnected} implies that $A\to A_+\simeq A\oplus A$ is also $0$-connected. Note that, by virtue of $A_+$ being an object of \textit{pointed} $\mb E_n$ $S$ algebras over $A$, that there is a retract $A\to A_+\to A$ and therefore $A_+\to A$ is $1$-connected by Proposition \ref{prop:retract}. By Proposition \ref{prop:excisionsuspension} we then have that $A_+\to \Omega_A\Sigma_A A_+$ is $2$-connected. 

Now let $W$ be a set with $m$-elements and consider the cocartesian $m$-cube $F\co\mathcal{P}(W)\to \Alg_{\mb E_n}(S)_{/A}$ with  $A_+$ as the initial vertex, $\Sigma^{m-1} A_+$ as the final vertex, and $A$ for every other vertex. Given a subset $V\subseteq W$ with $|V|=k$, the homotopy colimit over the diagram $\mathcal{P}(V)-V$ is $\Sigma_A^{k-1}A_+$. By iterating Proposition \ref{prop:pushoutconnectivity} we see that $\Sigma^{k-1}_AA_+\to \Sigma^{k-1}_AA\simeq A$ is $k$-connected, so the subcube of $F$ associated to $V$ is $k$-cocartesian (except when $V=W$ in which case it is $\infty$-coCartesian). In other words, in the notation of Definition \ref{def:excisive}, $k_V=|V|$ whenever $|V|<|W|$ and $k_V=\infty$ when $V=W$. Thus the entire cube $F$ must be $K$-cartesian where $K$ is the minimum $\{-|W|+\Sigma_{V\in\lambda}\left(|V|+1\right)=|\lambda|:|\lambda|>1\}\cup \{\infty\}$. Therefore $K=2$. Thus $A_+\to \Omega^m_A\Sigma_A^m A_+$ is at least $2$-connected for all $m\geq 0$. 

Note that because the connectivity structure on $\Alg_{\mb E_n}(S)_{/A}$ is lifted along a forgetful functor $\Alg_{\mb E_n}(S)_{/A}\to \LMod_S$ which preserves sifted colimits, it is compatible with, in particular, sequential colimits. Thus $A_+\to \colim_{m}\Omega^m_A\Sigma^m_A A_+\simeq A\oplus \mb L_A$ is $2$-connected, and so $A\to A\oplus\mb L_A$ is $0$-connected.
\end{proof}

The following is well known to experts, but the authors could not readily find a reference using the $\infty$-categorical language of this paper.

\begin{lem}\label{lem:basechangedcotangent}
	Let $f\co B\to A$ be a morphism of $\mb E_n$ $S$-algebra spectra for $S$ a commutative ring spectrum. Then the value of the functor $(\Sigma^\infty_+)_A\co Alg^{\mb E_n}_{/A}\to \Sp(Alg^{\mb E_n}_{/A})$ applied to $f$ is equivalent to $\loday{S}{A}\otimes_{\loday{S}{B}}\mathbb{L}^{\mb E_n}_B$. 
\end{lem}

\begin{proof}
 We check that $\loday{S}{A}\otimes_{\loday{S}{-}}\mathbb{L}^{\mb E_n}_{(-)}$ is left adjoint to $\Omega^\infty_A$. Let $f^\ast\vdash f_\ast$ be the push-pull adjunction between $Alg^{\mb E_n}_{/A}$ and $Alg^{\mb E_n}_{/B}$ and $F^\ast\vdash F_\ast$ its associated stabilization (cf.~\cite[Proposition 5.4.1]{lurie--goodwillieI}). We check that $F_\ast\mathbb{L}_B^{\mb E_n}\simeq\loday{S}{A}\otimes_{\loday{S}{B}}\mathbb{L}^{\mb E_n}_B$ and $\Sigma^\infty_{A,+}B$ corepresent the same functor in $\Sp(Alg_{A}^{\mb E_n})$. For any $M\in \Sp(Alg_{/A}^{\mb E_n})$ we have the following string of (natural in $M$) equivalences of mapping spaces:
 \begin{align*}
 Alg_{/A}^{\mb E_n}(f_\ast B,\Omega^\infty_A M)&\simeq Alg_{/B}^{\mb E_n}(B,f^\ast \Omega^\infty_A M)\\
&\simeq Alg^{\mb E_n}_{/B}(B,\Omega^\infty_B F^\ast M)\\
&\simeq \Sp(Alg_{/B})(\mathbb{L}_{B}^{\mb E_n},F^\ast M)\\
&\simeq \Sp(Alg_{/B})(F_\ast\mathbb{L}_{B}^{\mb E_n},M).
 \end{align*}
 
 The first equivalence is by definition, because $f_\ast$ simply postcomposes with $f$ and $f^\ast$ is the pullback functor. The second equivalence follows from \cite[6.2.2.14 (3)]{lurie-higheralgebra} combined with \cite[7.3.1.5]{lurie-higheralgebra}. The third equivalence is the usual adjunction between a category and its stabilization (as well recalling the equivalence $\Sigma^\infty_{B,+}B\simeq \mathbb{L}_B^{\mb E_n})$). And the final equivalence is simply given by the extension/restriction of scalars adjunction between $F_\ast$ and $F^\ast$. Finally, we apply the fact that $Alg_{/A}^{\mb E_n}(f_\ast B,\Omega^\infty_AM)$ is naturally equivalent to $\Sp(Alg_{/A}^{\mb E_n})(\Sigma^\infty_{A,+}B,M)$.
\end{proof}

\begin{rmk}
	In Lemma \ref{lem:basechangedcotangent} we take for granted that the equivalences $\Mod_A^{\mb E_n}\simeq \LMod_{\loday{S}{A}}$ and $\Mod_B^{\mb E_n}\simeq \LMod_{\loday{S}{B}}$ are compatible with the relevant base change functors. However this follows from the description of those equivalences given, for instance, in \cite[Proposition 2.23]{francis-tangentcomplex}. 
\end{rmk}

\begin{thm}\label{thm:skeletaalgebras}
  Suppose that $S$ is a connective commutative ring spectrum such that $\pi_0(S)$ is a principal ideal domain, and that $A$ is a connective $\mb E_n$ $S$-algebra such that $S \to A$ is $1$-connected. Then given a cellular construction of a (minimal) $(N+n)$-skeleton
  \[
    Z \to H\pi_0(A) \tens_{\loday{S}{A}} A
  \]
  as $H\pi_0(A)$-modules, there is a corresponding cellular construction of a (minimal) $N$-skeleton $B \to A$ of $\mb E_n$-algebras. In the case that $Z\to H\pi_0(A)\tens_{\loday{S}{A}}A$ is minimal, there is an equivalence
  \[
    Z \simeq H\pi_0(A) \tens_{\loday{S}{B}} B
  \]
  of $H\pi_0(A)$-modules over $H\pi_0(A) \tens_{\loday{S}{A}} A$.
\end{thm}

\begin{proof}
  There is a composite adjunction
  \[
    \Alg^{\mb E_n}_{/A} \rightleftarrows \Sp\left(\Alg_{/A}^{\mb E_n}\right) \simeq \LMod_{\loday{S}{A}} \rightleftarrows \LMod_{H\pi_0(A)}
  \]
  between $\mb E_n$-algebras over $A$ and $H\pi_0(A)$-modules. The leftmost adjunction is the $\Omega^\infty_A\vdash \Sigma^\infty_A$ adjunction and the rightmost is the base-change adjunction between left $\loday{S}{A}$-modules and left $H\pi_0(\loday{S}{A})\simeq H\pi_0(A)$-modules (using Lemma \ref{lem:lodaypi0}). We will use this composite adjunction to pass a skeleton of $H\pi_0(A)\tens_{\loday{S}{A}} A$ back to a skeleton of $A$. 
  
   By using Theorem \ref{thm:cotangentcofiber} to write $\mb L_A^{\mb E_n}$ as $\Sigma^{-n}\cofib(\loday{S}{A}\to A)$ we see that the image of $A$ under the composite left adjoint above is
  \[
    H\pi_0(A) \tens_{\loday{S}{A}} \mb L^{\mb E_n}_{A} \simeq \Sigma^{-n} \cofib\left(H\pi_0 (A) \to H\pi_0(A) \tens_{\loday{S}{A}} A\right).
  \]
  Given a (minimal) $(N+n)$-skeleton $Z\to H\pi_0(A) \tens_{\loday{S}{A}} A$, Corollary \ref{cor:uniquelifts} implies that the map $H\pi_0(A)\to H\pi_0(A) \tens_{\loday{S}{A}} A$ lifts to $Z$ because $H\pi_0(A)$ is $0$-skeletal. By Proposition \ref{prop:pushoutconnectivity}, the induced map
  \[
    \cofib(H\pi_0(A)\to Z)\to \cofib\left(H\pi_0(A)\to H\pi_0(A)\tens_{\loday{S}{A}}A\right)
  \]
  is $(N+n)$-connected (or $(N+n+1)$ in the case of a minimal skeleton). By Proposition \ref{prop:pushoutskeletal}, $\cofib(H\pi_0(A)\to Z)$ is $(N+n)$-skeletal. Therefore, by stability, the resulting $n$-times desuspended map $\Sigma^{-n} \cofib(H\pi_0 A \to Z) \to H\pi_0(A) \tens_{\loday{S}{A}} \mb L^{\mb E_n}_{A}$ is a (minimal) $N$-skeleton.
  
  By Corollary~\ref{cor:PIDskeleton}, there exists a cellular construction of this (minimal) $N$-skeleton as an $H\pi_0(A)$-module using the cells $(\Sigma^{k-1} H\pi_0(A) \to \ast)_{1 \leq k \leq N}$. These cells are the images of the $0$-bounded cells $S \tens \mb E_n(S^{k-1}) \to S \tens \mb E_n(\ast)$ in the category of $\mb E_n$ $S$-algebras.
  
  We now excavate this cellular skeleton by first using the adjunction between left $H\pi_0(A)$-modules and left $\loday{S}{A}$-modules. Because $\pi_0(A)\cong\pi_0(\loday{S}{A})$ we have that the adjunction is subductive by Example \ref{exam:0connectedringmapgivessubductive}. Because the map $\loday{S}{A} \to H\pi_0(A)$ is a $1$-connected map of connective ring spectra, the left adjoint $LMod_{\loday{S}{A}}\to LMod_{H\pi_0(A)}$ reflects connectivity in the standard $t$-structure inherited from $\Sp$. Finally, because $0$-truncation preserves tensor products for connective spectra, we have that $\mb L_A^{\mb E_n}\to H\pi_0(A)\otimes_{\loday{S}{A}}\mb L_A^{\mb E_n}$ is $0$-connected.  So Theorem \ref{thm:excavateskeleta} allows us to excavate the (minimal) cellular skeleta of $H\pi_0(A)\otimes_{\loday{S}{A}}\mb L_A^{\mb E_n}$ to (minimal) cellular skeleta for $\mb L_{A}^{\mb E_n}$.
  
  We now excavate \emph{this} skeleton from $LMod_{\loday{S}{A}}\simeq \Sp(\Alg^{\mb E_n}_{/A})$ to $\Alg^{\mb E_n}_{/A}$. We will again use Theorem~\ref{thm:excavateskeleta}, but first need to establish that the hypotheses thereof hold in this case. Because the unit map $S\to A$ is $1$-connected, and the connectivity structure on $\Alg^{\mb E_n}_{/A}$ is lifted from that of  $\Alg^ {\mb E_n}$, we may assume that we have restricted our adjunction to $1$-connected objects of $\Alg^{\mb E_n}_{/A}$. It follows from Corollary \ref{cor:subductivealgebras} that the $\Sigma\dashv\Omega$-adjunction thereon is strongly subductive and thus by Proposition \ref{prop:subductivestabilization}, so is $\Sigma^\infty\dashv\Omega^\infty$. Now using Corollary \ref{cor:strongsubductivereflects} (along with Remark \ref{rmk:loopsinfinitypreservesconnectivity} and Theorem \ref{thm:stabilizationconnectivity}), we have that stabilization reflects connectivity. As a result, Theorem \ref{thm:excavateskeleta} allows us to excavate (minimal) cellular skeleta from $\Sp(\Alg^{\mb E_n}_{/A})$ to $\Alg^{\mb E_n}_{/A}$. 
  
Putting the pieces together, we have that any cellular construction of a (minimal) $(N+n)$-skeleton $Z\to H\pi_0(A)\otimes_{\loday{S}{A}}A$ can be excavated to a cellular construction of a (minimal) $N$-skeleton $B\to A$. Now notice that the composite left adjoint $\Alg^{\mb E_n}_{/A}\to \LMod_{H\pi_0(A)}$ preserves connectivity and skeletal objects (the latter by Proposition \ref{prop:adjointskeletal}). Therefore it also preserves minimal skeleta. Thus, if $Z$ is minimal, our description of $Z$ will follow from a computation of the image of $B\to A$ under the composite adjunction given at the beginning of the proof. From Lemma \ref{lem:basechangedcotangent} we deduce that \[B\mapsto H\pi_0(A)\tens_{\loday{S}{B}}\mathbb{L}_B^{\mb E_n}\simeq \Sigma^{-n}\cofib\left(H\pi_0(A)\to H\pi_0(A)\tens_{\loday{S}{B}}B\right)\]

 Therefore homotopy uniqueness of minimal skeleta (followed by $n$-fold suspension) gives an equivalence
  \[
     \cofib\left(H\pi_0(A) \to Z\right) \simeq \cofib\left(H\pi_0(A) \to H\pi_0(A) \tens_{\loday{S}{B}} B\right).
  \]
  
  First notice that by functoriality of the fiber sequence given in Theorem \ref{thm:cotangentcofiber} (cf.~for instance \cite[7.3.5.5]{lurie-higheralgebra}) we have a comutative diagram
  
    \[
  \xymatrix{
  	\loday{S}{B} \ar[r] \ar[d] & \loday{S}{A} \ar[d] \\
  	B \ar[r] &
  	A
  }
  \]
  of $\loday{S}{B}$-modules. From this we obtain another commutative diagram
  
      \[
  \xymatrix{
  	H\pi_0(A)\ar[rr]\ar[d]^\simeq & & H\pi_0(A)\ar[d]^\simeq\\
  	H\pi_0(A)\tens_{\loday{S}{B}}\loday{S}{B} \ar[r]  \ar[d] & H\pi_0(A)\tens_{\loday{S}{B}}\loday{S}{A} \ar[r]\ar[d]& H\pi_0(A) \tens_{\loday{S}{A}}\loday{S}{A} \ar[d] \\
  	H\pi_0(A)\tens_{\loday{S}{B}} B \ar[r] & H\pi_0(A)\tens_{\loday{S}{B}} A \ar[r]& H\pi_0(A)\tens_{\loday{S}{A}}A
  }
  \]
  in which the top (and middle) horizontal map is equivalent to the identity. Moreover, the left and right vertical composites are the ones obtained by tensoring $H\pi_0(A)$ over $\loday{S}{B}$ and $\loday{S}{A}$ respectively with the maps $\loday{S}{B}\to B$ and $\loday{S}{A}\to A$ arising in Theorem \ref{thm:cotangentcofiber}. It follows from the proof of Lemma \ref{lem:basechangedcotangent} that the bottom horizontal map is the map whose cofiber gives the ($n$-fold suspension of) the $N$-skeleton $\Sigma^{-n}\cofib\left(H\pi_0\to H\pi_0(A)\tens_{\loday{S}{B}}B\right)\to \Sigma^{-n}\cofib\left(H\pi_0(A)\to H\pi_0(A)\tens_{\loday{S}{A}}A\right)$ described above.
  
  Thus we have a commutative diagram 
      \[
  \xymatrix{
  	H\pi_0(A)\ar[r]\ar[d]^\simeq & H\pi_0(A)\tens_{\loday{S}{B}} B\ar[d]\ar[r]& \cofib\left(H\pi_0(A)\to H\pi_0(A)\tens_{\loday{S}{B}} B\right)\ar[d]\\ H\pi_0(A)\ar[r] &
  	H\pi_0(A)\tens_{\loday{S}{A}}A\ar[r] & \cofib\left(H\pi_0(A)\to H\pi_0(A)\tens_{\loday{S}{A}} A\right)
  }
  \]
  whose right vertical map is the $n$-fold suspension of our skeleton. Notice also that by continuing the exact sequences to the right, obtaining horizontal maps to $\Sigma H\pi_0(A)$, we may write the map $H\pi_0(A)\tens_{\loday{S}{B}}B\to H\pi_0(A)\tens_{\loday{S}{A}}A$ as a fiber, and thus, by Proposition \ref{prop:pullbackconnectivity}, is $(N+n)$-connected. Therefore, we may lift the skeleton $Z\to H\pi_0(A)\tens_{\loday{S}{A}}A$ along it to obtain a commutative diagram 
  
      \[
  \xymatrix{ &Z \ar[dr]\ar[d]&\\
  	H\pi_0(A)\ar[r]\ar[ur] & H\pi_0(A)\tens_{\loday{S}{A}}A
	& H\pi_0(A)\tens_{\loday{S}{B}}B\ar[l] }
  \]
  in which the left hand triangle is the one obtained by lifting $H\pi_0(A)$ along the $(N+n)$-connected map $Z\to H\pi_0(A)\tens_{\loday{S}{A}}A$, and the right hand triangle is the one obtained by lifting the $(N+n)$-skeleton $Z$ along the $(N+n)$-connected map $H\pi_0(A)\tens_{\loday{S}{B}}B$. 
  
  Fitting the above triangle into the diagram preceding it, and taking the relevant cofiber, we now have a commuting diagram in which all rows are exact:
  
    \[
  \xymatrix{
  	H\pi_0(A)\ar[d]^\simeq \ar[r]&Z\ar[d]\ar[r]&\cofib\left(H\pi_0(A)\to Z\right)\ar[d]\\
  	H\pi_0(A)\ar[r]\ar[d]^\simeq & H\pi_0(A)\tens_{\loday{S}{B}} B\ar[d]\ar[r]& \cofib\left(H\pi_0(A)\to H\pi_0(A)\tens_{\loday{S}{B}} B\right)\ar[d]\\ H\pi_0(A)\ar[r] &
  	H\pi_0(A)\tens_{\loday{S}{A}}A\ar[r] & \cofib\left(H\pi_0(A)\to H\pi_0(A)\tens_{\loday{S}{A}} A\right)
  }
  \]
  By Propositions \ref{prop:minimalretract} and \ref{prop:skeletaluniqueness}, the upper right vertical map in the above diagram must be the unique-up-to-homotopy equivalence between $\cofib\left(H\pi_0(A)\to Z\right)$ and $\cofib\left(H\pi_0(A)\to H\pi_0(A)\tens_{\loday{S}{B}}B\right)$. It follows that the middle vertical map, $Z\to H\pi_0(A)\tens_{\loday{S}{B}}B$ is also an equivalence. 
\end{proof}

\begin{rmk}
Note that, in the case that we are considering an $\mb E_1$ $S$-algebra $A$, there is an equivalence $THH^S(A,\pi_0(A))\simeq H\pi_0(A)\otimes_{\loday{S}{A}} A$. So Theorem \ref{thm:skeletaalgebras} allows us to lift minimal skeleta of the often simpler topological Hochschild homology of $A$ to minimal skeleta of $A$ itself, which we will use in the next sections.
\end{rmk}

\subsection{Remark on Thom spectra}

Recall the universal property of a Thom spectrum as a ring \cite{antolin-camarena-barthel-thom}: given a map of grouplike $\mb E_n$-spaces $f\co G \to GL_1(\mb S)$, maps $Mf \to R$ of associative rings are equivalent to nullhomotopies of the composite
\[
  B^nG \to B^nGL_1(\mb S) \to B^nGL_1(R).
\]
We will show that if $B^nG$ is $(k+n-1)$-skeletal, then $Mf$ is $k$-skeletal as an $\mb E_n$-algebra.

Suppose that $R \to S$ is a $k$-connected map of ring spectra for $k > 0$. The map $B^nGL_1(R) \to B^nGL_1(S)$ is $(k+n)$-connected by directly considering homotopy groups. If $B^nG$ is $(k+n-1)$-skeletal, then the map $B^nG \to \ast$ is a relatively $(k+n)$-skeletal map, and so any nullhomotopy of the map $B^nG \to B^nGL_1(S)$ lifts to a nullhomotopy of the map $B^nG \to B^nGL_1(R)$. By the universal property, this asserts that any map $Mf \to S$ lifts to a map $Mf \to R$, as desired.

\subsection{$T(n)$ as associative skeleta}

In this section we fix a prime $p\in\mb{Z}$. In \cite{ravenel-greenbook}, Ravenel introduces a sequence of homotopy commutative ring spectra called $T(n)$ and morphisms of homotopy commutative ring spectra $T(n)\to T(n+1)$ such that $T(0)\simeq \mb{S}_{(p)}$ and $\hocolim(T(n))\simeq \BP$, the Brown-Peterson spectrum. Ravenel computes that $\BP_\ast(T(n))\cong \BP_\ast[t_1,\ldots,t_n]\subset \BP_\ast \BP$ and the natural map $T(n)\to \BP$ is a homology isomorphism in degrees less than $\vert t_{n+1}\vert=2(p^{n+1}-1)$; hence it is $(2p^{n+1}-3)$-connected. In what follows, we show that this is in fact the inclusion of a minimal skeleton of $\BP$ in the category of $\mb E_1$-ring spectra.

We begin with some calculations.

\begin{prop}
  \label{prop:THHTn}
  We have
  \[
    \pi_* \THH(\BP; \mb Z_{(p)}) \cong \Lambda[\sigma t_1, \sigma t_2,
    \dots].
  \]
  If $T \to \BP$ is any map of connective $p$-local ring spectra inducing the inclusion $\mb Z_{(p)}[t_1\dots,t_n] = H_* T(n) \to H_* \BP$ on homology, then the map
  \[
    \pi_* \THH(T;\mb Z_{(p)}) \to \pi_* \THH(\BP;\mb Z_{(p)})
  \]
  is isomorphic to the inclusion of the subalgebra
  \[
    \Lambda[\sigma t_1, \dots, \sigma t_n],
  \]
  which is precisely those elements in degrees less than $2p^{n+1}-1$.
\end{prop}

\begin{proof}
  From \cite[Lemma 2.2]{thhell} we have an equivalence
  \[
    \THH(\BP;\mb{Z}_{(p)})\simeq
    H\mb{Z}_{(p)}\sma_{\BP\sma \BP^{op}}\BP\simeq
    H\mb{Z}_{(p)}\sma_{H\mb{Z}_(p)\sma \BP^{op}}H\mb{Z}_{(p)}. 
  \]
  It follows, from either \cite[IV, 4.1]{ekmm} or \cite[Corollary 2.3]{thhell} that we have a K\"unneth spectral sequence of signature
  \[
    \Tor_{\ast\ast}^{H_\ast(\BP^{op};\mb{Z}_{(p)})}(\mb{Z}_{(p)},\mb{Z}_{(p)})
    \Rightarrow\THH_\ast(\BP;\mb{Z}_{(p)}).
  \]
  The $E_2$-term reduces to
  \[
    \Tor_{\ast\ast}^{\mb Z_{(p)}[t_1,t_2,\dots]} (\mb Z_{(p)}, \mb
    Z_{(p)})
  \]
  which is an exterior algebra $\Lambda[\sigma t_1, \sigma t_2,\dots]$ with $\sigma t_i$ in total degree $2p^i-1$ and filtration $1$. The elements $\sigma t_i$ are permanent cycles for degree reasons, and the spectral sequence has multiplicative structure because $\BP$ admits an $\mb E_4$-ring structure \cite{basterra-mandell-BP}. The product structure then implies that all elements in the spectral sequence are permanent cycles, and there is no room for any hidden multiplicative extensions.

  Given such a map $T \to \BP$, the same spectral sequence for $\pi_* \THH(T;\mb Z_{(p)})$ maps injectively to the spectral sequence for $\THH_*(\BP; \mb Z_{(p)})$, and hence has no differentials or hidden extensions either. The resulting map is the inclusion
  \[
    \Lambda[\sigma t_1, \dots, \sigma t_n] \subset \Lambda[\sigma t_1,
    \sigma t_2, \dots].
  \]
  The highest degree of a nontrivial element on the left is
  \[
    |\sigma t_1 \sigma t_2 \dots \sigma t_n| = \sum_{i=1}^n (2p^i-1) =
    2p\tfrac{p^n-1}{p-1} - n < |\sigma t_{n+1}|
  \]
  and therefore the subalgebra consists precisely of those elements in degrees less than $2p^{n+1}-1 = |\sigma t_{n+1}|$.
\end{proof}

\begin{thm}\label{thm:T(n)areSkeleta}
  There exists an associative algebra stucture on $T(n)$ making it into a minimal $(2p^{n+1}-4)$-skeleton of $\BP$ as a $p$-local associative algebra.
\end{thm}

\begin{proof}
  By Proposition~\ref{prop:EMskeletal}, the $\mb Z_{(p)}$-module $\THH(\BP;\mb Z_{(p)})$ has a minimal $2p^{n+1}-3$ skeleton.

  Therefore, by Theorem~\ref{thm:skeletaalgebras}, this minimal $(2p^{n+1}-3)$-skeleton can be excavated: there is a minimal $(2p^{n+1}-4)$-skeleton $T \to \BP$ as an associative algebra, whose $\THH$ coincides through degree $(2p^{n+1}-3)$. By Proposition~\ref{prop:THHTn}, this forces
  \[
    \pi_* \THH(T; \mb Z_{(p)}) \cong
    \Lambda[\sigma t_1, \dots, \sigma t_n].
  \]

  The map $T \to \BP$ is then $(2p^{n+1}-3)$-connected. This means $t_1,\dots,t_n \in H_* \BP$ are in the image, so $H_* T \to \mb Z_{(p)}[t_1,\dots,t_n]$ is surjective and an isomorphism through degrees $(2p^{n+1}-4)$. Moreover, because we know the homology of the tangent $T$ lifts the tangent complex, the $\THH$ spectral sequence
  \[
    \Tor^{H_* T} (\mb Z_{(p)}, \mb Z_{(p)}) \Rightarrow \Lambda
  \]
  for the homology of the tangent complex must be isomorphic to the corresponding spectral sequence for the subalgebra $\mb Z_{(p)}[t_1,\dots,t_n]$; this can only be true if $\mb Z_{(p)}[t_1,\dots,t_n] \to H_* T$ is an isomorphism.

  Consider the commutative diagram of inclusions and retractions of $p$-local ring spectra \cite[p. 217]{ravenel-greenbook}:
  \[
    \xymatrix{
      &  &
      X(p^n)_{(p)} \ar[r] \ar[d] &
      T(n) \ar[d] \\
      T \ar@{.>}[urr] \ar[r] &
      \BP \ar[r] &
      \MU_{(p)} \ar[r] &
      \BP
    }
  \]
  The dotted lift exists because the map $X(p^n) \to \MU$ is a $(2p^n-1)$-connected map of associative algebras, and hence we get a composite map $T \to T(n)$ of $p$-local spectra. On homology this becomes the commutative diagram
  \[
    \xymatrix{
      &  &
      \mb Z_{(p)}[x_1,\dots,x_{2p^n-2}] \ar[r] \ar[d] &
      \mb Z_{(p)}[t_1,\dots,t_n] \ar[d] \\
      \mb Z_{(p)}[t_1,\dots,t_n] \ar@{.>}[urr] \ar[r] &
      \mb Z_{(p)}[t_1,t_2,\dots] \ar[r] &
      \mb Z_{(p)}[x_1,x_2,\dots] \ar[r] &
      \mb Z_{(p)}[t_1,t_2,\dots]
    }
  \]
  This makes the map $T \to T(n)$ into a homology isomorphism, and thus an equivalence, of spectra over $\BP$.
\end{proof}

\begin{rmk}
Theorem \ref{thm:T(n)areSkeleta} bears some similarity to \cite[Corollary 13]{beardsley-relativeThomSpectra} and \cite[Theorem 2]{beardsley-Xncells} in which it is shown that the spectra $X(n)$ and their $p$-localizations $X(n)_{(p)}$, of which the spectra $T(n)$ are wedge summands, can be constructed from $X(n-1)$ and $X(n-1)_{(p)}$ via attachments of $\mb E_1$-$X(n-1)$-algebra cells. Those results however use the construction of $X(n)$ as a Thom spectrum in a crucial way, and $T(n)$ is not a Thom spectrum. Those results also do not immediately imply that $X(n)$ is a skeleton of $MU$ as $\mb E_1$-ring spectra (in contrast to the results for $T(n)$ given above).
\end{rmk}

\subsection{$Y(n)$ as associative skeleta}

Similarly, we recall that at $p=2$ there is a spectrum $Y(n)$ with a map $Y(n) \to H\mb F_2$ whose homology maps isomorphically to a subalgebra of the dual Steenrod algebra:
\[
  H_* (Y(n); \mb F_2) \cong \mb F_2[\xx_1, \dots, \xx_n].
\]
By contrast with the previous case, there is a known associative multiplication on $Y(n)$ \cite{mahowald-thomcomplexes}. Specifically, the space $\Omega S^3$ is equivalent to a CW-complex with one cell in each even degree, via the James construction; the spectrum $Y(n)$ is the Thom spectrum of the composite map of loop spaces
\[
  \Omega (\Omega S^3)^{(2^{n+1}-2)} \to \Omega^2 S^3 \to BO.
\]

We will now show that the spectrum $Y(n)$ is a skeleton of $H\mb F_2$ as an associative ring spectrum.

\begin{prop}
  We have
  \[
    \pi_* \THH(H\mb F_2) \cong \mb F_2[u].
  \]
  If $T \to H\mb F_2$ is any map of connective ring spectra inducing the inclusion $\mb F_2[\xi_1\dots,\xi_n] = H_* Y(n) \to H_* H\mb F_2$ on homology, then the map
  \[
    \pi_* \THH(T;\mb F_2) \to \pi_* \THH(H\mb F_2)
  \]
  is isomorphic to the inclusion of those elements in degrees less than or equal to $2^{n+1}-2$.
\end{prop}

\begin{proof}
  This follows because the map on $E_2$-pages of K\"unneth spectral sequences is the inclusion of all classes in total degree less than or equal to $2^{n+1}-2$, and the K\"unneth spectral sequence for $\THH(H\mb F_2)$ degenerates.
\end{proof}

\begin{cor}
  The spectrum $Y(n)$ is $(2^{n+1} - 3)$-skeletal as an associative algebra.
\end{cor}

\begin{proof}
  The top degree where the topological Hochschild homology of $Y(n)$ is nontrivial is
  \[
    |\sigma \xx_1 \dots \sigma \xx_n| = \sum_{i=1}^n 2^i =
    2^{n+1}- 2.
  \]
  By Proposition~\ref{prop:THHskeletality} and \ref{prop:EMskeletal}, the spectrum $Y(n)$ is $(2^{n+1} - 3)$-skeletal as a ring spectrum.
\end{proof}

The map $Y(n) \to H\mb F_2$ is $(2^{n+1}-2)$-connected. Therefore, we arrive at the following conclusion.

\begin{thm}
  The spectrum $Y(n)$ is a minimal $(2^{n+1}-3)$-skeleton of $H\mb F_2$ as an associative algebra.
\end{thm}

\begin{rmk}
  We can give an alterative proof: instead of using topological Hochschild homology as with $T(n)$, we can use the description of $Y(n)$ as a Thom spectrum.
\end{rmk}

\subsection{Involutions}

\begin{prop}
  Suppose that $R$ is an $\mb E_2$-algebra and that $T$ is a minimal $k$-skeleton of $R$. Then there is an involution $\lambda\co T \to T^\op$ over $R$.
\end{prop}

\begin{proof}
  Any $\mb E_2$-algebra is equivalent to its own opposite algebra. The composite
  \[
    T^\op \to R^\op \too{\sim} R
  \]
  then makes $T^\op$ into a minimal $k$-skeleton of $R$. However, uniqueness of minimal skeleta from Proposition~\ref{prop:skeletaluniqueness} implies that there is a canonical equivalence $T \to T^\op$ of algebras mapping to $R$.
\end{proof}

\begin{rmk}
  The self-equivalence $R \to R^\op$ becomes, after forgetting the $\mb E_2$-structure, the identity self-map of $R$. By contrast, the involution on $T$ may not be homotopic to the identity map, and hence this does not prove homotopy commutativity of $T$.
\end{rmk}

\begin{cor}
  The spectrum $T(n)$ is equivalent to $T(n)^\op$ as an associative algebra with a map to $\BP$.
\end{cor}

In other words, we have an \emph{involution} $\lambda\co T(n) \to T(n)^\op$ of the algebra $T(n)$.

\begin{cor}
  The spectrum $Y(n)$ is equivalent to $Y(n)^\op$ as an associative algebra with a map to $H\mb F_2$.
\end{cor}

\begin{rmk}
  The analogous result should be true at odd primes, with a more involved calculation in topological Hochschild homology.
\end{rmk}

\bibliography{masterbib}
\end{document}